\documentclass[12pt,a4paper]{article}
\usepackage[english]{babel}
\usepackage{fullpage}
\usepackage[dvipsnames]{xcolor}
\usepackage{psfrag}
\usepackage{pgfplots}
\usepackage{subfigure}
\usepackage{amsmath}
\usepackage{fancybox}
\usepackage{amsfonts}
\usepackage{amsthm}
\usepackage{amssymb}
\usepackage{relsize}
\usepackage{comment}
\usepackage{graphicx}
\usepackage{tikz}
\usepackage{enumerate}
\usepackage{MnSymbol}
\usepackage{url}
\usepackage{algorithm}
\usepackage{algorithmic}
\usepackage{tabularx}
\usepackage{soul}

\newif\ifdraftversion

\usetikzlibrary{external}
\theoremstyle{plain}
    \newtheorem{theorem}{Theorem}[section]
    
    \newtheorem{lemma}[theorem]{Lemma}
    
    \newtheorem{assumption}[theorem]{Assumption}
\theoremstyle{definition}
    
    \newtheorem{remark}[theorem]{Remark}
\theoremstyle{remark}

\usepackage{bm}
\usepackage{xspace}

\usepackage{color}
\definecolor{miverde}{RGB}{128,255,0}
\definecolor{minaranja}{RGB}{255,128,1}

\DeclareMathOperator{\supp}{supp}

\usepackage[normalem]{ulem} 

\newcommand{\QQ}{\textsf{Q}}

\newcommand\BB{\mathbb B}
\newcommand{\B}{{\mathcal B}}

\newcommand\M{\mathcal M}

\newcommand\T{\mathcal T}
\newcommand\VV{\mathbb V}

\newcommand{\I}{\mathcal I}

\newcommand{\bF}{{\bf F}}

\def\vv{\mathbf{v}}

\newcommand{\disc}{\textrm{disc}}

\newcommand{\uu}{\tilde{u}}

\usepackage{authblk}

\definecolor{dgreen}{rgb}{0.5,0.2,1}
\definecolor{gold}{rgb}{.7,.5,0}

\definecolor{dred}{rgb}{0.92,0,0}
\definecolor{dgreen}{rgb}{0,0.6,0}

\let\hat\widehat
\let\tilde\widetilde

\newcommand{\norm}[1]{{\left\vert\kern-0.25ex\left\vert #1 
    \right\vert\kern-0.25ex\right\vert}}

\ifdraftversion
\usepackage{xfp} 
\newcommand{\convergenceslope}[6] 
{
 \draw [color=black,mark=none,#6]
 (axis cs: #1, #3)  -- node [midway, #5] {$h^{#4}$} (axis cs: #2, \fpeval{exp((#4) * ln((#2)/(#1)) + ln(#3))});
}
\fi

\def\M3AS{Math.\ Models\ Methods\ Appl.\ Sci.}

\newcommand{\review}[1]{{\color{black}#1\color{black}}}

\begin{document}

\title{Isogeometric Analysis on V-reps: first results}
\author[1]{Pablo Antolin}
\author[1,2]{Annalisa Buffa}
\author[2]{Massimiliano Martinelli}
\affil[1]{\small \'Ecole Polytechnique F\'ed\'erale de Lausanne, Institute of Mathematics, Lausanne, Switzerland.}
\affil[2]{\small Istituto di Matematica Applicata e Tecnologie Informatiche 
`E. Magenes' (CNR), Pavia, Italy.}

\maketitle

\begin{abstract} 
Inspired by the introduction of Volumetric Modeling via volumetric representations (V-reps) by Massarwi and Elber in 2016, in this paper we present a novel approach for the construction of isogeometric numerical methods for elliptic PDEs on trimmed geometries,  seen as a special class of more general V-reps.  We develop tools for approximation and local re-parameterization of trimmed elements for three dimensional problems,  and we provide a theoretical framework that fully justifies our algorithmic choices. We validate our approach both on two and three dimensional problems, for diffusion and linear elasticity.
\end{abstract}

\begin{quote}\small
\end{quote}

\section{Introduction}
Back in 2005, T.\ Hughes and co-authors introduced isogeometric analysis (IGA) \cite{hughes2005} with the promise to alleviate the burden of conversion between computer-aided design (CAD) geometries and finite element (FE) computational domains. IGA has been one of the most successful ideas in the last decades in computational mechanics, and there are really a huge amount of contributions to the field, also because splines have been  proved  to be a very powerful tool for the approximation of solutions of partial differential equations (PDEs), \cite{beirao2014}. Indeed, IGA has been a tremendous success since 2005 with a wide range of applications (see for instance the special issue \cite{CMAME-IGA2017}), and it is becoming a mature method: its mathematical analysis is now well understood \cite{beirao2014}, fast assembly and solvers exist today \cite{MR3761002,MR3811621} and strategies for adaptive refinement with a sound mathematical theory are now available, see \cite{1078-0947_2019_1_241} and the references therein. 

Nevertheless, the construction of  geometric representations suitable to IGA, from CAD geometries,  remains the major unsolved issue and IGA  is indeed a  very efficient strategy on ``simple'' geometries (i.e., on 2D or $2.5$D geometries \cite{Bletz15,Bletzinger2017,Guo18}) 
but not for general three dimensional problems.  For 3D objects the conversion of a CAD geometry into a geometry on which  isogeometric methods are defined (IGA-ready) is \review{challenging and has many similarities with  the conversion to a FE domain: } its automation is far from being at the state of the art.  The main reason why this problem is still open is that CAD geometric descriptions are based two main ingredients:  \emph{i)}  they represent only the boundary of geometries and not their interior, i.e., B-reps \cite{braid_designing_1974,requicha_solid_1992} ; \emph{ii)} they allow for  boolean operations among spline surfaces, primitives and so on. The result is a complex boundary representation that by no means corresponds to a valid mesh. For two dimensional problems, this issue was successfully addressed for the first time in \cite{benson15}, and then by many other authors.
Interest readers may refer to the review \cite{Marussig2018} and the many references therein. 

As a matter of fact, this has limited the use of isogeometric techniques for complex three dimensional domains. The only viable approach has been, until now, the use of spline basis on unstructured hexahedral (or tetrahedral) meshes, which has been object of several contributions \review{\cite{li_spline-based_2012,wang2012,Wei17,xia_isogeometric_2017}}.
In this context, the presence of extraordinary points and edges make the construction of regular B-spline functions preserving accuracy extremely challenging (see \cite{Wei18,kapl_isogeometric_2018} and the references therein).

On the other hand, a  new paradigm has recently been introduced by Elber et al.\ in \cite{Massarwi201636}, according to which geometries are described as V-reps, i.e., via  volumetric representations. V-reps are defined as a generalization of B-reps as boundaries of V-reps are always B-reps. Geometries are described thorough the geometric representation of the volume they occupy and the basic building blocks for V-reps are trivariate (as opposed to bivariate) representations. Boolean operations are handled at the level of volumes instead of surfaces. Loosely speaking, V-reps are a class of volumes in which volumetric boolean operations are allowed and produce V-reps.

As a matter of fact, V-reps are not IGA-ready geometries as the handling of boolean operations at the level of simulation of PDEs  is not at the state of the art, and is very much linked with the numerical treatment of trimming. 
In this paper, we initiate the construction of accurate isogeometric  methods on V-reps. 
In particular, we focus on the design of a numerical methodology for the simulation of elliptic problems on computational domains described in terms of V-reps, and the main rule we adopt in our design is the following:  we wish to construct numerical methods that uses the V-rep representation without 
 asking for meshing or global re-parameterization, while we will make use of local high-order re-parameterizations. 

Our work is related to previous efforts in the construction of numerical methods on trimmed domains and of immersed finite elements approaches. 
In particular, we should mention the Finite Cell approach, which was successfully applied in the isogeometric framework in \cite{schillinger2012,Rank2013,kudela_efficient_2015,Rank2017} and the references therein;
\review{the CutFEM and CutIGA methods \cite{burman_cutfem:_2015,Elfverson2018}; as well as other high-order unfitted Galerkin discretizations, e.g.\ \cite{lehrenfeld_analysis_2018}}.
Subdivision techniques are used in \cite{cirak1} and boundary corrections are proposed in \cite{cirak2}. 

\review{These families of immersed methods present numerous similiarities with the approach presented in this and previous works (see, e.g.\ \cite{benson15,Marussig2018}, for the case of 2D geometries).
Indeed, those unfitted methods define the solution discretization by means of a background mesh that is completely disconnected from geometry boundary; while, following an isoparametric paradigm, the approach that we present here considers the underlying V-rep's spline space as a basis for the solution discretization. 
Nonetheless, the results discussed in this work, namely, the presented error estimate and the construction of high-order local re-parameterizations,  can be used directly also in the case of finite element and isogeometric immersed methods.}

We say that a volume $\Omega$ is a  \emph{non-conforming multipatch trivariate volume} (nCMTV), if it is constructed as a collection of trivariates (patches), $T_1, \ldots T_k$ such that: 
\begin{itemize}
\item For each $T_\ell$, there is  a  spline diffeomorphism $\bF_\ell: \hat T \to T_\ell $ , $\ell=1,\ldots, k$ ($\hat T$ stands for the parameter space); 
 \item $\bar{T}_i\cap \bar{T}_j$ is either empty or is the image of a full face $\hat{f}_i$ of $\hat T$ for $T_i$ and a full face  $\hat{f}_j$ of $\hat T$ for $T_j$;
\item the parameterizations $F_i \vert_{\hat{f}_i} $ and $F_j\vert_{\hat{f}_j}$ may differ.
\end{itemize}

If the parameterizations $F_i \vert_{\hat{f}_i} $ and $F_j\vert_{\hat{f}_j}$ are identical, then we say that $\Omega$ is a conforming multipatch trivariate volume (CMTV).  
 Following  \cite{Massarwi201636}, volumetric representations  are obtained as the closure of the  nCMTV volumes  with respect to the boolean operations of union, intersection and their combinations. 

In this first contribution we consider only domains generated by subtraction (or intersection).
We assume that we are given a collection of nCMTVs  $\Omega_0, \Omega_1 ... \Omega_N$, and  we consider  the computational domain $\Omega$  defined as follows: 
\begin{equation}
  \label{eq:Omega}
   \Omega = \Omega_0 \setminus \bigcup_{i=1}^N \bar{\Omega}_i\ .
\end{equation}
Finally, for the sake of simplicity we will assume that $\Omega_0$ has indeed conforming meshes.
If this is not the  case, non conforming multipatch interfaces should be handled via mortaring techniques \cite{Brivadis15}.  
The aim of this paper is to propose algorithms to tackle the isogeometric simulation of linear elliptic problems on such a computational domain, without any re-parameterization and within the isoparametric paradigm.
As we concentrate of Neumann problems only, after reminding how the main approximation properties of splines ensure approximability of solutions on $\Omega$, we concentrate on the main issues: the definition of an integration formula on trimmed elements, and the study of the conditioning and pre-conditioning of the underlying linear system.
We discuss the approximation of the trimmed boundaries and provide approximation estimates which are validated against numerical benchmarks. 

Our contribution is organized as follows: in Section \ref{sec:theory} we defined our model problem and corresponding iso\-geo\-me\-tric discretization, together with our assembly algorithms and a brief discussion about the linear system conditioning.
In Section \ref{sec:experiments} we will provide numerical testing, we discuss both academic tests for validation of our algorithms, and also one rather complex application in linear elasticity.
We draw our conclusions in Section \ref{sec:conclusions}.

\section{V-reps as computational domain for PDEs}
\label{sec:theory}
Given the domain $\Omega$ defined in \eqref{eq:Omega}, we think of it as the domain of definition of elliptic partial differential equations. Indeed, we consider two model problems: the Laplace problem as model of  diffusion phenomena and compressible linear elasticity as a simplified model of the deformation of solids. 
 
The boundary of the domain  $\Omega$ defined  in \eqref{eq:Omega}  naturally splits in two parts: a non-trimmed part of boundary $\partial \Omega\cap  \partial \Omega_0$, 
and a trimmed boundary  $\partial \Omega \setminus \partial \Omega_0$.  

We denote by $\Gamma_D$ a connected subset (with Lipschitz boundary) of $\partial \Omega$.  For the time being, we assume that $\Gamma_D \subset \partial \Omega_0 \cap\partial \Omega$: this means that we allow for essential boundary conditions only on the non-trimmed part of the boundary. For the sake of simplicity, we assume that $\Gamma_D \neq \emptyset$. If this is not the case, all what follows remain valid, once the definition of spaces are adapted and suitable compatibility conditions of data are satisfied. 
 As usual, we set $\Gamma_N = \partial \Omega \setminus \bar{\Gamma}_D$. 

We define: 
\[ V=\{ u\in H^1(\Omega)^k \ :\  u\vert_{\Gamma_D}  =0 \}\,, \] 
where $k$ is either $1$ for diffusion problems or $3$ for elasticity problems. For both problems, we assume to have a bilinear form $a(\cdot,\cdot)$ acting on $V\times V $ and verifying:
\begin{itemize}
\item Continuity: there exists a $M>0$ such that $ a(u,v) \leq M \|u\|_V \| v\|_V$;
\item Coercivity: there exists a $\alpha>0$ such that $ a(u,u) \geq  \alpha \|u\|^2_V $. 
\end{itemize}

The continuous problem to solve is formulated as follows and admits a unique solution:
\begin{equation}
  \label{eq:varform}
 \text{Find }u\in V\  : \    a(u,v) = \int_\Omega f\,v + \int_{\Gamma_N}  g \, v  \qquad \forall v\in V.
\end{equation}
If needed, we will denote by $A$ the differential operator in strong form. I.e., $\langle Au,v\rangle_{V'\times V} = a(u,v)$. 
\begin{remark}
Essential boundary conditions can be imposed only on the non-trimmed part of boundary because in the discrete setting described below, the imposition of essential boundary condition on the trimmed part of the boundary ask for special care \cite{burman2012,burman2014,buffa_minimal_2019}. The design of robust techniques to deal with this issue in the framework of trimmed isogeometric analysis is beyond the scope of the present paper. Contributions in this direction can be found in \cite{hansbo,burman1,Marussig17}.
\end{remark}

\subsection{The isogeometric method on V-reps}
\label{sec:iga-on-Vreps}
Thanks to the construction above, the master volume $\Omega_0$ is described by $n_0$ trivariates $T_{0,1}, \ldots T_{0, n_0}$ and $n_0$ parameterizations $\bF_{0,1}, \ldots, \bF_{0,n_0}$. These parameterizations are  constructed on open knots vectors that are defined in the parameter trivariate space $\hat{T}$. We denote by $h_j$ the diameter  of the largest knot span in $T_{0,j}$. 

On each patch composing $\Omega_0$, we have a collection of  B-splines defined on such open knot vectors:
\begin{equation}
  \label{eq:spline-sp-ptaches}
  \BB_{0,j} := \{ B_{0,j,k} = \hat{B}_{0,j,k}\circ  \bF^{-1}_{0,j},   \quad  k\in \I_{0,j} \}  \;; 
\end{equation}
where $ \I_{0,j} $ denotes the running indices on the B-spline functions on the patch $T_{0,j}$. 

 We denote by $\BB^{\disc}(\Omega_0) := \bigcup_{j=0}^{n_0}  \BB_{0,j}  $ and by $\VV_h^\disc(\Omega_0) = \mathrm{span} \{ \BB^{\disc}(\Omega_0)  \}$, where we have set $h=\max_j h_j$. 
From now on, and for the sake of simplicity, we assume that $\Omega_0$ is a CMTV. 
We also construct the space $\vv_h(\Omega) = \VV_h^\disc(\Omega_0)  \cap H^1(\Omega_0)$. It is well known that a basis for this space can be constructed starting from $\BB^{\disc}(\Omega_0) $ and identifying the indices of functions corresponding to coincident control points at the interfaces. In what follows, we suppose to have a set of indices $\I_0:=\{ 0,1,\ldots, I_0 \}$ and a collection of B-splines functions $\BB(\Omega_0) $, indexed by $\ell\in \I_0$, that span $\VV_h(\Omega)$. The generic basis function is denoted as $\B_{0,\ell} \,, \; \ell \in \I_0$. Finally we denote by $\VV_h(\Omega_0) $ the space spanned by these functions. 

Mapping knot surfaces through the mappings $\bF_{0,j}$, we obtain a mesh, $\T_h({\Omega_0})$ which is, indeed, a curvilinear hexahedral partition of   $\Omega_0$. As $\Omega_0 $ is a CMTV, this partition is a conforming decomposition of the domain $\Omega_0$. 

\medskip

Refinement on such a mesh can be performed by knot insertion, and for the sake of our subsequent analysis, we suppose we have a family of refined meshes (obtained by dyadic refinement of all non empty knot spans)  and corresponding refined basis functions. Such a refinement will not change the geometry of the volume $\Omega_0$ but only its representation. In this way, we will assume to have a family of meshes  $\T_h({\Omega_0})$, of spaces $\VV_h(\Omega_0)$ indexed with $h$.  
The dependence on $h$ will be omitted in what follows when it does not matter. 

As our computational domain is $\Omega$ and not $\Omega_0$, we define $\T_h(\Omega)$ as follows:
\begin{equation}
  \label{eq:TOmega}
  \T_h(\Omega)  : =   \{ Q : \quad  \forall Q\in \T_h({\Omega_0})  \ :\ Q\cap \Omega \neq \emptyset    \}.
\end{equation}
Note that, by an abuse of notation,  we could also write: $  \T_h(\Omega) = \T_h({\Omega_0}) \cap \Omega $. 

For further use, among the elements in  $\T_h(\Omega) $, we distinguish two families: 
\begin{itemize}
\item $Q \in  \T_h(\Omega) $ such that ${Q} \cap \Omega = Q$, and we denote their collection as $\T^{int}_h(\Omega)  $;
\item $Q \in  \T_h(\Omega)  $ such that ${Q} \cap \Omega \neq  Q$, and we denote their collection as $\T^\Gamma_h(\Omega)  $; when $ Q\in \T^\Gamma_h(\Omega)  $, we will also make use of  $\mathsf{Q} = Q \cap \Omega$. 
\end{itemize}

The indices $\ell \in \I_0$, such that $ \supp\{ \B_{0,\ell}\} \cap \Omega \neq \emptyset $, are denoted by $\I_0^{act}$, whereas \review{$\# \I_0^{act}$ stands for its cardinality}.
We have the following: 
\begin{lemma} It holds:
  $$ \dim(\VV_h(\Omega)) = \# \I_0^{act} . $$
\end{lemma}

\begin{proof}
This is an immediate consequence of the local linear independence of B-splines: i.e., their restriction  to any element $Q$ is a collection of $(p+1)^3$ linear independent functions representing polynomials of degree $p$ in each coordinate direction.
\end{proof}

Thanks to our assumption on $\Gamma_D$, i.e. $\Gamma_D \subset  \partial \Omega_0 \cap \partial \Omega$,  we can define the space: 
\begin{equation}
  \label{eq:spDC}
  \VV_h({\Omega, \Gamma_D}) =  \{  v\in \VV_h(\Omega)  \ :\ v\,\vert_{\Gamma_D} =0 \};
\end{equation}
For both spaces $\VV_h({\Omega, \Gamma_D}) $ and   $\VV_h(\Omega) $, we also have the following approximation property: 
\begin{lemma}  \label{lem:approx}
 There exist  continuous  interpolation operators $\Pi_h: L^2(\Omega) \to \VV_h(\Omega)$ and $\Pi_{h}^{(D)}:L^2(\Omega) \to \VV_h(\Omega, \Gamma_D)$ such that:
 \begin{align}
   \| v - \Pi_h(v) \|_{H^s(\Omega)} & \leq C h^{t-s} \|   v  \|_{H^t(\Omega)}  \qquad \forall v\in H^{t}(\Omega)\,, \\
 \| v - \Pi_h^{(D)} (v) \|_{H^s(\Omega)} & \leq C h^{t-s} \|   v  \|_{H^t(\Omega)}  \qquad \forall v\in H^{t}(\Omega) \ :\ v_{\vert \Gamma_D} =0\,,
 \end{align}
for all $s=0,1$ and $1  \leq t \leq p+1. $
 
\begin{proof}
Given a $v \in H^{s+1}(\Omega)$, we denote by $\tilde{\cdot}$ the Sobolev extension operator as defined in \cite[theorem 5, page 181]{Stein70} that continuously extends any function defined in $\Omega$ to functions defined in $\mathbb{R}^3$ and verifies: 
\[ \| \tilde{v}  \|_{H^t(\mathbb{R}^3) }  \leq C \| v \|_{H^t(\Omega)}  \qquad t\geq 0\,,\]
where the constant $C$ does not depend on $t$.  
Now, given a $v \in H^{t}(\Omega)$,  the local quasi-interpolant defined in \cite{beirao2014}, or the one defined in \cite{buffa2016b}, can be applied to $\tilde{v}$ and the result follows. A similar idea applies to the case with Dirichlet boundary conditions, thanks to the assumption $\Gamma_D \subset \partial \Omega_0 \cap \partial \Omega$. 
 \end{proof}
\end{lemma}

We are now in the position to write our first discrete problem. We set  $V_h(\Omega) = \{ \VV_h(\Omega,\Gamma_D)^k \} $ (as before, $k=1$ or $3$),  and we wish to solve:
\begin{equation}
  \label{eq:varform-discr}
 \text{Find } u_h\in V_h(\Omega)\  : \    a(u_h,v_h) = \int_\Omega f\ v_h  + \int_{\Gamma_N} g\, v_h \qquad \forall v_h\in V_h (\Omega)\,.
\end{equation}
As $V_h(\Omega) \subset V $, both continuity and coercivity hold true, which implies that the discrete problem is well posed and that the solution depends continuously on the data.  

As a matter of fact, two types of problems may appear when trying to construct the linear system for the problem \eqref{eq:varform-discr}: 
\begin{enumerate}
\item We need to be able to compute integrals as $\int_{Q\cap \Omega}  \xi  $ where $\xi$ is a regular function (say a multiplication of splines, regular coefficients, Jacobians of transformations) but the integration domain is not known analytically for all $Q\in  \T^\Gamma_h(\Omega)  $.
\item The contribution of all those basis functions whose support intersects the computational domain only for a small portion may generate bad conditioning of the stiffness matrices. 
\end{enumerate}

Clearly the first observation is crucial since, as it is, the assembly of the stiffness matrix for the problem \eqref{eq:varform-discr} is not possible. The next two subsections are devoted to the design of an algorithm to compute integrals accurately. 
The latter problem is discussed in Section \ref{sec:conditioning}.

\subsection{Numerical integration error estimate}
\label{sec:integr_error_analysis}
In order to compute each contribution to the stiffness in an accurate way, we propose to: 
\begin{itemize}
\item[a)] approach each boundary $\gamma= Q\cap\partial \Omega$, for  $ Q \in \T_h^\Gamma(\Omega)$, with a suitable piecewise polynomial approximation surface (or line in 2D), that we call $\gamma_h$. \label{sec:integr_error_analysis_a}
\item[b)]  Construct a re-parameterization  of each approximated element which allows for the definition of a suitable integration formula. 
\end{itemize}

 In the sequel of this section we prove that if the approximation at \review{a)} verifies certain assumptions, then the accuracy of the isogeometric method is preserved.  In the next subsection instead, we provide algorithms to construct such (piecewise) polynomial  approximations together with  a strategy to construct local accurate integration formulae. 

 Given a $Q  \in \T_h^\Gamma(\Omega) $,   we denote by  $ \gamma = Q\cap  \partial \Omega $ and by $\gamma_h$ its approximation of degree $r$ (when needed, a subindex $Q$ will be used to identify the element to which $\gamma$ and $\gamma_h$ belong) .   In what follows we assume that, for all elements $Q\in \T_h^\Gamma(\Omega)$, the construction of $\gamma_h$ fulfills the following Assumption, with constants that do not depend on the specific $Q$. 

\begin{assumption}
\label{ass:approx-domain} We assume that 
\begin{enumerate}
\item $\gamma_h \subset \bar{Q}$
\item  there exists a finite collection of local charts so that both $\gamma$ and $\gamma_h$ can be described as local graphs. More precisely, 
there exists a finite number $n_Q$ of local systems of coordinates $(X_i, Y_i)$, of corresponding  neighborhoods $\Delta_i \subseteq \mathbb{R}^{n-1} $ and, for each of those, two functions $\psi_i :\Delta_i \to \mathbb{R}$ and $\psi_{h,i}: \Delta_i \to  \mathbb{R}$:
\begin{align}
  \label{eq:localchart1}
\gamma_i=\{ (X_i, \psi(X_i)) \ ; \ X_i \in \Delta_i\} \subseteq \gamma\,, \qquad \gamma_{h,i}=\{ (X_i, \psi_h(X_i)) \ ; \ X_i \in \Delta_i\} \subseteq \gamma_h . 
\end{align}
Moreover, there exists a constant C such that 
\[ \bigcup_{i=1}^{n_Q}  \gamma_i = \gamma \;;\qquad  \bigcup_{i=1}^{n_Q}  \gamma_{h,i} = \gamma_h  \;; \quad  \sum _{i=1}^{n_Q} \mathrm{meas} (\gamma_i) \leq C\,  \mathrm{meas} (\gamma)   \;; \quad \sum _{i=1}^{n_Q} \mathrm{meas} (\gamma_{h,i}) \leq C  \mathrm{meas} (\gamma_h) .  \]
\item  the following approximation property holds: 
\[ \| \psi - \psi_h\|_{W^{m,\infty}(\Delta_i)} \leq C h^{r+1-m} \qquad m=0,1\,,   \]
where $r$ is the degree of the approximation $\gamma_h$ of $\gamma$. 
\end{enumerate}
\end{assumption} 

 Reminding that $\mathsf{Q} = Q\cap \Omega$, we denote now by $\mathsf{Q}_h$ the element obtained from $\mathsf{Q}$ by replacing its boundary $\gamma \subset \partial \Omega$ by the corresponding $\gamma_h$. 

Similarly, we denote by $\Omega_h$ the domain obtained by replacing each $\mathsf{Q}$, associated with $Q\in \T_h^\Gamma(\Omega)$ with the corresponding $\mathsf{Q}_h$; by $\Gamma_{h,N}$ the approximation of the Neumann boundary;  and by $a_h(\cdot,\cdot)$ and $L_h(\cdot{}) $ the bilinear and right hand side forms for the problem \eqref{eq:varform-discr} when $\Omega$ is replaced by $\Omega_h$. Moreover, we assume that the data of the problem $f$ and $g$ are restrictions of $\tilde{f}$ and  $\tilde{g}$ that are defined in $\mathbb{R}^3$ and belong to the corresponding Sobolev space (i.e., for $f\in H^s(\Omega)$ then $\tilde{f} \in H^s(\mathbb{R}^3)$, and if $g\in H^{s-1/2}(\Gamma_N)$ then $\tilde{g}\in H^s(\mathbb{R}^3)$).
\review{As for the solution $u$, by a little abuse of notation, we denote by $\tilde{u}$ its Sobolev extension, as defined in the well known \cite[Theorem 5, p.181]{Stein70}. We will make use of the following 
\[ \| \tilde{u} \|_{H^s(\Omega_h) }  \leq C  \| {u} \|_{H^s(\Omega) }  \qquad \forall \; s>0\,,\]
where the constant $C$ does not depend on $s$. }
Thus, instead of \eqref{eq:varform-discr}, we solve the following problem:
\begin{equation}
  \label{eq:varform-discr2}
 \text{Find } \uu_h\in V_h(\Omega_h)\  : \    a_h(\uu_h,v_h) = L_h(v_h) = \int_{\Omega_h} \review{\tilde{f}}\ v_h  + \int_{\Gamma_{h,N}} \review{\tilde{g}}\, v_h \qquad \forall v_h\in V_h (\Omega_h)\,.
\end{equation}

In the sequel we prove that, under the Assumption \ref{ass:approx-domain}, the problem~\eqref{eq:varform-discr2} is an optimal approximation of \eqref{eq:varform}. Note that, for the sake of simplicity, we assume that on each approximated element, the integration is performed exactly. Of course, this is not the case in practice but, instead, the quadrature is chosen in order for the integration error to verify the classical assumptions of the finite element theory. Taking into account this error here would make our proofs unreasonably tedious, without adding much.  We refer the interested reader to \cite{Ciarlet2002}. 

Our approach is largely inspired by the papers \cite{Elliott2,Elliott1} and uses the technical results obtained in \cite{cermak83}. We basically adapt their approach to our problem. 

We are going to prove the following approximation result. 
\begin{theorem}
  \label{th:approx}
Let $\Omega  \subset \mathbb{R}^n$, $n=2, 3$  be a computational domain defined as in \eqref{eq:Omega}.
Let $u$ be the solution of \eqref{eq:varform}. 
We assume that $u\in H^{t+1}(\Omega)$, with $t+1>\frac{n}{2}+2$, and we denote by $\uu$ its Sobolev extension. Let $\uu_h$ be the solution of \eqref{eq:varform-discr2}.  Then, under the Assumption \ref{ass:approx-domain}, it holds: 
\begin{equation}
  \label{eq:optimal}
   \| \tilde{u} - \uu_h\|_{H^1(\Omega_h)} \leq  C  h^{t} \| u \|_{H^{t+1}(\Omega)}  + C h^{r} \| \uu \|_{H^{n/2+2+\alpha}(\Omega)} \qquad 2\leq t \leq p\,, \quad  0\leq r \leq p\,,
\end{equation}
where $p$ denotes the degree of the spline approximation, $r$ the degree of the local re-parameterization, $n=2, 3$ is the dimension of the ambient space, and $\alpha>0$. Moreover, it holds: 
\begin{equation}
  \label{eq:l2-est}
   \| \tilde{u} - \uu_h\|_{L^2(\Omega_h)} \leq C h^{t+1} \| u \|_{H^{t+1}(\Omega)}  + C h^{r+1} \| \uu \|_{H^{n/2+2+\alpha}(\Omega)} \qquad 2\leq t \leq p \,,\   0\leq r \leq p\;. 
\end{equation}
\end{theorem}

\begin{remark}
    Indeed, the error due to the geometry approximation depends upon the $W^{2,\infty}(\Omega)$ norm of the (extension of) the solution $u$, and we have used here the continuous Sobolev embedding $H^{n/2+\alpha}(\Omega) \subset L^\infty(\Omega)$.  
\end{remark}

\begin{proof}
  By application of the Strang Lemma, see e.g., \cite{Ciarlet2002}, and using coercivity and continuity,  we can control the distance between $\uu_h$ and any function $w_h\in  V_h(\Omega_h)$ (sometimes below we also use $\chi_h =  \uu_h - w_h$) as follows:
\begin{equation}
\begin{split}
    \label{eq:1}
    \alpha \| \uu_h - w_h\|^2_{H^1(\Omega_h)}  & \leq a_h(\uu_h - w_h, \uu_h - w_h)  = a_h(\uu_h,\chi_h) - a_h(w_h,\chi_h) = L_h(\chi_h) - a_h(w_h,\chi_h) \\
& \leq a_h(\uu-w_h, \chi_h) + \| \chi_h\|_{H^1(\Omega_h)}\ \sup_{\xi_h\in V_h(\Omega_h) } \frac{L_h(\xi_h) - a_h(\uu,\xi_h)}{\| \xi_h\|_{H^1(\Omega_h)} }.  
\end{split}
\end{equation}
We need to estimate the supremum in the right hand side. 
\begin{equation}
  \label{eq:2}
    a_h(\uu, \xi_h) = \int_{\Omega_h} A \uu\,  \xi_h + \int_{\Gamma_{h,N}} (\nabla \uu \cdot n_h  )\, \xi_h\,,  \qquad L_h(\xi_h) = \int_{\Omega_h} \tilde{f} \,\xi_h + \int_{\Gamma_{h,N}} \tilde{g} \,\xi_h \,,
\end{equation}
where $n_h$ denotes the unit vector normal to $\Gamma_{h,N}$, and then 
\begin{equation}
  \label{eq:3}
L_h(\xi_h) - a_h(\uu,\xi_h) = \int_{\Omega_h\setminus \Omega} (\tilde{f} - A \uu ) \xi_h + \int_{\Gamma_{h,N}} ( \tilde{g} - \nabla \uu \cdot n_h)\,\xi_h = I + II\,,
\end{equation}
where we have used that $A \uu = f$ pointwise for almost every $x\in \Omega$. The two terms $I$ and $II$ are studied separately, and we will make use of a few important estimates which hold true under the Assumption \eqref{ass:approx-domain}. We use the estimates shown in, e.g., \cite[Lemma 4.1]{Elliott1}, but consider their local versions. For all $Q \in  {\cal T}_h^\Gamma$ that for all $w\in H^1(Q)$, letting  $\QQ= Q\cap \Omega $ and $\QQ_h$ is its approximation,  it holds:
\begin{equation}
  \label{eq:elliott}
  \| w\|_{L^2(\QQ_h\setminus \QQ)}\leq C (h^{r+1} | w|_{H^1(\QQ\setminus \QQ_h)} + h^{\frac{1}{2}(r+1)} \| w\|_{L^2(\partial\QQ_h)})\,,
\end{equation}
where the constant $C$ does not depend on the specific $Q$, neither on the way $\gamma$ or $\gamma_h$ cuts the element $Q$.
Keeping in mind that the trace operator is continuous from $L^2(\partial \QQ_h)$ to $H^1(Q)$, with a continuity constant that does not depend on the way $\gamma_h$ cuts the element $Q$, sometimes we make use of the following estimate of the right hand side: 
\begin{equation}
  \label{eq:elliott_2}
  h^{r+1} | w|_{H^1(\QQ_h\setminus \QQ)} + h^{\frac{1}{2}(r+1)} \| w\|_{L^2(\partial\QQ_h)} \leq C  h^{\frac{1}{2}(r+1)}  \| \review{w} \|_{H^1(Q)}\,.
\end{equation}

We also use then the estimates proved by \v{C}erm\'ak in \cite[Lemma 3.2 and Lemma 3.3]{cermak83}, and we apply them again locally to a $\QQ$.  For all $w\in W^{2,\infty}(Q)$, verifying $\nabla w \cdot n - g =0 $ on $\gamma= Q \cap \partial \Omega$, it holds: 
\begin{equation}
  \label{eq:cermak}
 \left\| \nabla w \cdot n_h - \tilde{g} \right\|_{L^2(\gamma_{h})} \leq C h^{r} \| w\|_ {W^{2,\infty}(Q)}\,. 
\end{equation}

We start by estimating the term $I$ in the r.h.s.\ of \eqref{eq:3}, using \eqref{eq:elliott}:
\begin{equation*}
  \begin{aligned}
      I & \leq  \sum_{Q\in {\cal T}_h^\Gamma} \| \tilde{f} - A \uu \|_{L^2(\QQ_h\setminus\QQ)} \|\xi_h\|_{L^2(\QQ_h\setminus\QQ)} \\
& \leq C \sum_{Q\in {\cal T}_h^\Gamma}  \big[  h^{r+1} |\tilde{f} - A \uu  |_{H^1(\QQ_h\setminus \QQ)} + h^{\frac{1}{2}(r+1)} \| \tilde{f} - A \uu \|_{L^2(\partial\QQ_h)} \big]\cdot \\
& \qquad \qquad \big[   h^{r+1} |\xi_h  |_{H^1(\QQ_h\setminus \QQ)} + h^{\frac{1}{2}(r+1)} \| \xi_h \|_{L^2(\partial\QQ_h)} \big]\,.
  \end{aligned}
\end{equation*}

By \eqref{eq:elliott_2}, we can bound the r.h.s.\ as:
\begin{equation}
  \label{eq:9}
  \begin{aligned}
     I & \leq C  \sum_{Q\in {\cal T}_h^\Gamma}   (h^{r+1} \| \xi_h\|_{L^2(\partial\QQ_h)}  + h^{\frac{3}{2}(r+1)} \| \xi_h\|_{H^1(\QQ_h)} ) \| \tilde{u} \|_{H^3(Q)}\\ 
& \leq C  (h^{r+1} \| \xi_h\|_{L^2(\partial\Omega_h)}  + h^{\frac{3}{2}(r+1)} \| \xi_h\|_{H^1(\Omega_h)} ) \| u \|_{H^3(\Omega)}\\ 
& \leq C h^{r+1} \|  \xi_h\|_{H^1(\Omega_h)}  \| u \|_{H^3(\Omega)}\,,
  \end{aligned}
\end{equation}
where we have used the continuity of the trace operator from $L^2(\partial \Omega_h) $ to $H^1(\partial \Omega_h)$, and that its continuity constant does not depend on $h$, but may depend on the constants in the Assumption \ref{ass:approx-domain}.  

We should now estimate the second term in the r.h.s.\ of \eqref{eq:3}, using \eqref{eq:cermak} \review{and the Sobolev embedding theorem. We have, for $\alpha>0$:}
\begin{equation}
  \label{eq:EsII}
  \begin{aligned}
II & \leq  \sum_{Q\in {\cal T}_h^\Gamma} \|  \nabla \uu - \tilde{g}\|_{L^2(\gamma_h)} \| \xi_h\|_{L^2(\gamma_h)}   \leq C \sum_{Q\in {\cal T}_h^\Gamma}  h^{r}  \|\tilde{u} \|_{W^{2,\infty}(Q)}  \| \xi_h\|_{L^2(\gamma_h)}\\
& \leq C  h^{r}  \|\tilde{u} \|_{W^{2,\infty}(\Omega)}  \| \xi_h\|_{H^{1}(\Omega_h)} \leq C  h^{r}  \|\tilde{u} \|_{H^{n/2+2+\alpha}(\Omega_h)} \| \xi_h\|_{H^{1}(\Omega_h)} \,.
 \end{aligned}
\end{equation}

Thus, using \eqref{eq:9}, \review{\eqref{eq:EsII}} and \review{the continuity of the Sobolev extension,} we obtain, 
\[ | L_h(\xi_h) - a_h(\uu,\xi_h) |  \leq  C h^{r} \| \xi_h\|_{H^{1}(\Omega_h)}  \|{u} \|_{H^{n/2+2+\alpha}(\Omega)}\,, \]
which, together with \eqref{eq:1}, proves \eqref{eq:optimal}. 
We are left now with the proof of \eqref{eq:l2-est} and we are going to use a simple and classical duality argument. Our procedure is very similar to \cite{bramble94} and \cite{Elliott1}. 

It holds 
\begin{equation}
  \label{eq:10}
  \|\uu - \uu_h\|_{L^2(\Omega_h)} = \sup_{\eta\in L^2(\Omega_h)} \frac{\int_{\Omega_h} (\uu - \uu_h) \eta}{\|\eta\|_{L^2(\Omega_h)}}\,.
\end{equation}

We solve now, 
\[ A z = {\eta}\,, \qquad z = 0 \text{ on } \Gamma_D \,,\ \nabla z\cdot n= 0 \text{ on } \Gamma_N. \]
As a consequence of the regularity assumption we have done on the solution $u$, $z\in H^2(\Omega)$, and it holds: 
\[\| z\|_{H^2(\Omega)} \leq C \| \eta\|_{L^2(\Omega)}. \]
Now, we compute, for the right hand side of \eqref{eq:10}
\begin{equation*}
  \begin{aligned}
     \int_{\Omega_h} (\uu - \uu_h) \tilde{\eta} & = \int_{\Omega_h} (\uu - \uu_h) \tilde{A z} =  \int_{\Omega} (\uu - \uu_h) \tilde{A z} + \int_{\Omega_h\setminus\Omega} (\uu - \uu_h) \tilde{A z} = \\
&  = a(u - \uu_h,  z) + \int_{\Omega_h\setminus\Omega} (\uu - \uu_h) \tilde{A z}=  I + II \,,
  \end{aligned}
\end{equation*}
where we have used that $u - u_h =0 $ on $\Gamma_D$, and  $\nabla z\cdot n= 0 \text{ on } \Gamma_N.$ 
Now, we estimate the two terms $I$ and $II$ separately.
First, the term $II$ is simple to estimate by \eqref{eq:elliott} and \eqref{eq:elliott_2}:
\[ II \leq C h^{\frac{1}{2}(r+1)} \|\uu-\uu_h\|_{H^1(\Omega_h)} \| \tilde{Az}\|_{L^2(\Omega_h)} \leq C h \|\uu-\uu_h\|_{H^1(\Omega_h)} \|\eta\|_{L^2(\Omega)}\,, \] 
as $r\geq1$.
Now, we concentrate on the term $I$. It holds:
\[
\begin{aligned} 
a(u - \uu_h,  z)  = a_h(\uu-\uu_h,\tilde{z}) - (a_h-a)(\uu-\uu_h,\tilde{z})\,,
 \end{aligned}
\]
where the second term can easily be bounded with \eqref{eq:elliott} and \eqref{eq:elliott_2} as 
\[ (a_h-a)(\uu-\uu_h,\tilde{z})  \leq C h \|\uu-\uu_h\|_{H^1(\Omega_h)} \|z \|_{H^2(\Omega)}\,. \] 

Now, the first term  can be written as, for $z_h\in \VV(\Omega_h)$ being the best fit of $z$:
\begin{equation}
  \label{eq:5}
  a_h(\uu-\uu_h,\tilde{z})  = a_h(\uu-\uu_h, \tilde{z} - z_h) + a_h(\uu-\uu_h,  z_h) 
 \leq C h \|\uu-\uu_h\|_{H^1(\Omega_h)} \|z\|_{H^2(\Omega)} + |L_h(z_h)-a_h(\uu,z_h)|\,.
\end{equation}
The second term can be estimated, similarly as above, and with arguments similar to the ones used in \cite[Lemma 4.2]{Elliott1}, as: 
\[  |L_h(z_h)-a_h(\uu,z_h)| \leq C h^{r+ 1} \|z\|_{H^2(\Omega)} \| \uu\|_{W^{2,\infty}(\Omega)}\,.   \]
Thus, the estimate \eqref{eq:l2-est} holds true. 
\end{proof}

\begin{remark}
\label{rem:subopt}
  Going through the various steps in the proof, it is clear that the estimates \eqref{eq:optimal} and \eqref{eq:l2-est} may be improved, with the respect to the degree $r$ of the parameterization, as in few places we use suboptimal Sobolev inclusions, or estimates. Indeed, as we will see in the next section, numerical validation seems to suggest that the convergence is indeed slightly (less than half a order) faster for the examples we discuss.
\end{remark}

\subsection{Element-wise re-parameterizations for trimmed domains} 
\label{sec:local-rep}
As it was stated above, the use of boolean operations in the definition of the computational domain $\Omega$
makes the computation of integrals over $\Omega$ far from trivial.

Integrals are computed over every single element
and their contributions are added to the global result, as $\int_{\Omega}  \xi = \sum_{Q\in\Omega} \int_{Q\review{\cap\Omega}}\xi$.
Thus, we will distinguish the integrals over the two different families of elements $\T^{int}_h(\Omega)$ and $\T^\Gamma_h(\Omega)$:
\begin{equation}
\int_{\Omega} \xi = \mathlarger{\sum}_{Q\in\T^{int}_h(\Omega)} \int_{Q}\xi + \mathlarger{\sum}_{Q\in\T^{\Gamma}_h(\Omega)} \int_{Q\cap\Omega}\xi\,.
\end{equation}

The integrals over the B\'ezier elements $Q\in \T^{int}_h(\Omega)$ (including the integrals on their boundaries)
are computed, in a standard way, using a tensor-product Gauss-Legendre quadrature rule.
On the other hand, the integrals over the cut B\'ezier elements $Q\in \T^{\Gamma}_h(\Omega)$ require a special procedure.

In this section we present a novel technique for computing integrals over the
active part $\mathsf{Q}$ of cut B\'ezier elements. We describe our approach  in 3D, but it works verbatim in 2D.
Our technique has a lot in common with the integration strategy used in the Finite Cell methods (FCM), \cite{Rank2017,kudela_smart_2016}. Even if FCM is normally used as an immersed method, i.e., with $\mathbf{F}$ being the identity map, it has been generalized to non-trivial $\mathbf{F}$ in \cite{rank_shell_2011}, for thin walled structures.

While a comparison of the two approaches is beyond the scope of this paper, we anticipate that our method is likely to have a lower complexity, i.e. requires less evaluations, for coarse meshes. Indeed, when the mesh is coarse, and the trimmed B\'ezier elements have complex shapes, we do not perform any dyadic splitting which, in principle, generates several quadrature points.  
On the other hand, we expect the two techniques to be comparable for fine enough meshes.  

\review{An alternative approach for a high-order re-parameterization of the trimmed B\'ezier elements has been recently proposed in \cite{massarwi_volumetric_2019} for the case of 3D V-reps in which $\mathbf{F}$ is a non-trivial map.}

In order to compute the integrals over cut B\'ezier elements the first step is to identify
them and to obtain a V-rep of each of them.
This operation, denominated \emph{slicing}, is detailed in Section \ref{sec:slicing}.
In a second stage, a high-order re-parameterization is created for
every cut B\'ezier element. This re-parameterization will help us in building
quadrature rules adapted to the active part of every element.
This operation is described in Section \ref{sec:local-rep-bezier}.

Even if the techniques presented in this section are valid for both 2D and 3D domains, we restrict our discussion to the 3D case.
In addition, and for the sake of simplicity, hereinafter we assume that the domain $\Omega_0$ is 
composed by a single trivariate $T$, whose associated spline diffeomorphism is $\bF: \hat T \to T$.
This simplification does not constitute a limitation and does not prevent the methodology
described below to be applied to cases in which $\Omega_0$ is a multipatch trivariate.

\subsubsection{Slicing}\label{sec:slicing}
The slicing process consists in splitting the domain $\Omega$ into B\'ezier elements.
The goal of this operation is to identify the two families of elements
$\T^{int}_h(\Omega)$ and \review{$\T^{\mathsf{Q}}_h(\Omega)  : =   \{ \mathsf{Q}\ :\ \mathsf{Q}=Q\cap\Omega  \ :\ \forall Q\in \T^{\Gamma}_h(\Omega)\}$}.

The splitting is performed by recursively bisecting $\Omega$ along all the knot surfaces \review{in the parametric domain $\hat T$.
We denote \emph{knot surface} as the composition with $\bF$ of a knot plane of the parametric domain $\hat T$.
A knot plane is defined as a plane perpendicular to a parametric direction $d$ at a fixed knot value $d_k$.

The slicing procedure followed in this work is closely related to the one described in \cite[Section 3.1]{massarwi_volumetric_2019}, being its main steps summarized in Algorithm \ref{alg:slicing}.
Additionally, the procedure workflow is also illustrated in Figure \ref{fig:reparam:1}, where the knot surfaces are represented in yellow color.
}
\begin{figure}[t!]
\centering
\includegraphics[width=0.95\textwidth]{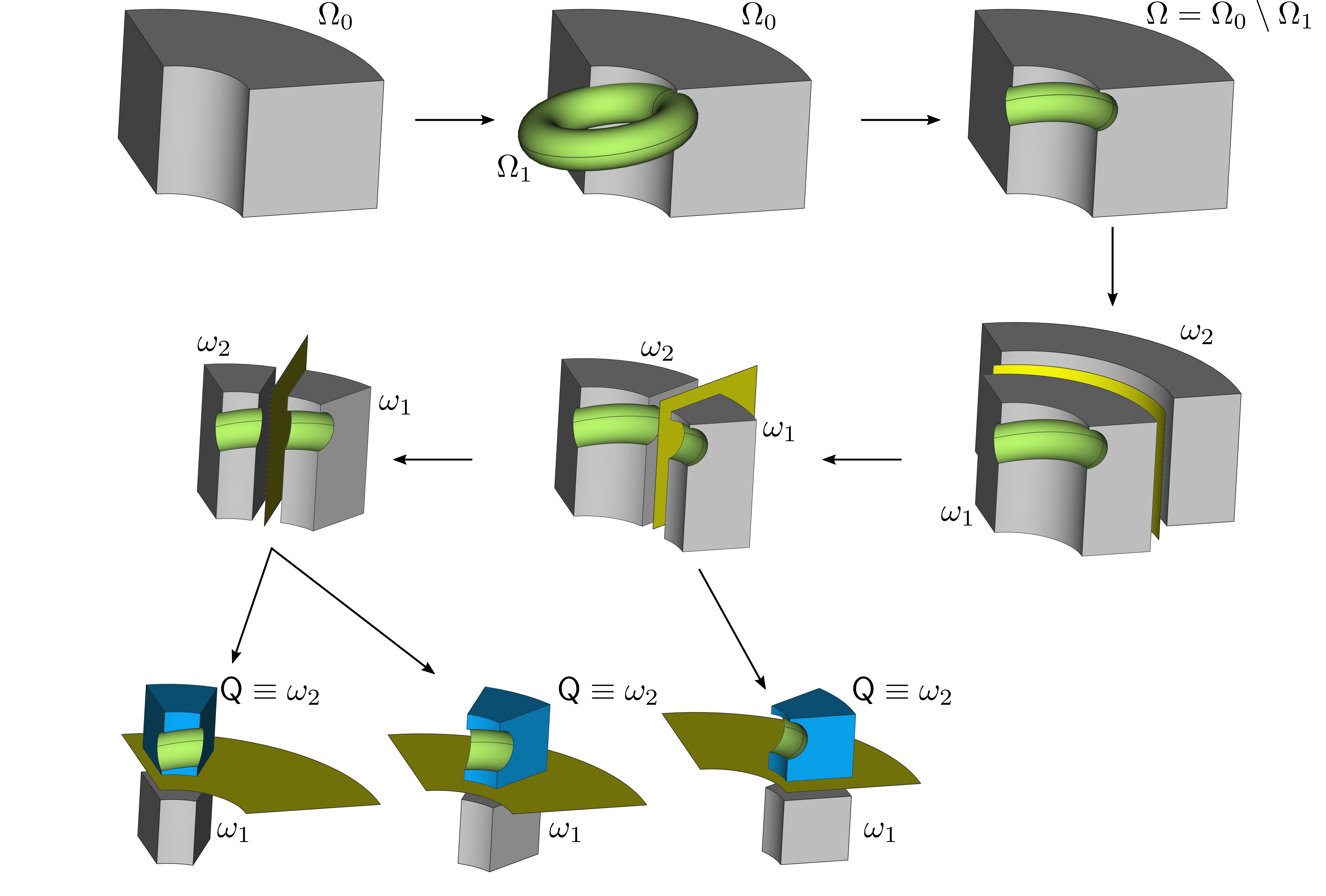}
\caption{Workflow of the creation of the domain $\Omega$ by boolean operations and slicing process.
As output of the slicing, the non-cut and cut elements are identified, and V-rep parameterizations of the
active part $\mathsf{Q}$ of the latter are obtained (in blue).}
\label{fig:reparam:1}
\end{figure}

\review{
We initially consider the V-rep $\omega$ ($\omega\equiv\Omega$ if it corresponds to the full domain) that has an associated B-spline tensor-product trivariate $T_\omega$ and a boundary representation $\gamma_\omega$ (a set of trimmed surfaces).
$\omega$ is then split according to a knot surface perperdicular to the longest parametric direction of the domain $\hat T_\omega$.
As a result of this splitting operation, denoted as {\bf SubdvWithKnotSrf} in Algorithm \ref{alg:slicing}, two new V-reps are obtained, $\omega_1$ and $\omega_2$, one on each side of the surface.
The new V-reps are subsequently sliced by applying the same algorithm recursively.

The slicing procedure finishes when all the B\'ezier elements have been identified as members of $\T^{int}_h(\omega)$ or $\T^{\mathsf{Q}}_h(\omega)$.
As mentioned above, the goal of the slicing operation is to identify cut and non-cut
B\'ezier elements. Therefore, if at a certain step the algorithm finds a V-rep that is not trimmed,
i.e., $\gamma_\omega$ coincides with the boundary of its associated trivariate $\partial T_\omega$, the slicing procedure is stopped for that particular V-rep.
All the B\'ezier trivariates contained in that particular V-rep belong to $\T^{int}_h(\omega)$ (steps 14-16 of Algorithm \ref{alg:slicing}).

In this work, this procedure has been implemented using the geometric kernel Open CASCADE \cite{OpenCASCADE}
for the splitting of the B-reps and IRIT \cite{IRIT} for managing trivariates.
}

\begin{algorithm}
\review{
	\textbf{Input}:\\
		\begin{tabularx}{\textwidth}{ll}
		$\omega=\{T_\omega,\gamma_\omega\}$: &V-rep composed of a single B-spline tensor-product trivariate $T_\omega$\\ &and its boundary representation $\gamma_\omega$ (a set of trimmed surfaces).
	\end{tabularx}
	\vspace{1mm}
	
	\textbf{Output}:\\
	\begin{tabularx}{\textwidth}{ll}
	$\T^\mathsf{Q}_h(\omega)$:& Set of trimmed B\'ezier trivariates (V-reps).\\
	$\T^{int}_h(\omega)$:& Set of non-trimmed active B\'ezier trivariates.
	\end{tabularx}
	\vspace{1mm}
	
\textbf{Algorithm}:
	\begin{algorithmic}[1]
 		\STATE $\T^\mathsf{Q}_h(\omega)$ := $\emptyset$,$\quad$ $\T^{int}_h(\omega)$ := $\emptyset$
		\IF {$\omega\neq\emptyset$}
		\IF {$\omega$ is trimmed ($\partial T_\omega \neq \gamma_\omega$)}
		  \IF {$T_\omega$ is a B\'ezier trivariate}
    		\STATE $\T^\mathsf{Q}_h(\omega)$ := $\omega\quad$ ($\mathsf{Q}\equiv\omega$)
		  \ELSE
  		\STATE $d_k, d$ := Find a (approx.) middle internal knot, $d_k$, along the largest\\ parametric direction $d$ ($u,v$ or $w$) of $\hat T_\omega$
	    \vspace{1mm}
  		\STATE [$\omega_1$, $\omega_2$] := \textbf{SubdvWithKnotSrf}(${\omega},d_k,d$)
	    \vspace{1mm}
  		\STATE [$\T^\mathsf{Q}_h(\omega_1)$, $\T^{int}_{h}(\omega_1)$] := \textbf{Slicing}(${\omega_1}$),$\quad$ [$\T^\mathsf{Q}_h(\omega_2)$, $\T^{int}_{h}(\omega_2)$] := \textbf{Slicing}(${\omega_2}$)
	    \vspace{1mm}
    	\STATE $\T^\mathsf{Q}_h(\omega)$ := $\T^\mathsf{Q}_h(\omega) \bigcup \T^\mathsf{Q}_h(\omega_1) \bigcup \T^\mathsf{Q}_h(\omega_2)$
	    \vspace{1mm}
    	\STATE $\T^{int}_h(\omega)$ := $\T^{int}_h(\omega) \bigcup \T^{int}_{h}(\omega_1) \bigcup \T^{int}_{h}(\omega_2)$
		  \ENDIF
		\ELSE
  		\FORALL[$\forall$ B\'ezier elements $T_{\omega,i}$ in $T_\omega$]{$T_{\omega,i} \in T_{\omega}$}
    		\STATE $\T^{int}_h(\omega)$ := $\T^{int}_h(\omega) \bigcup T_{\omega,i}$
  		\ENDFOR
		\ENDIF
		\ENDIF
 	  \RETURN $[\T^\mathsf{Q}_h(\omega)$, $\T^{int}_h(\omega)]$
	\end{algorithmic}
	\caption{\bf Slicing(${\omega}$): slicing procedure of a V-rep ${\omega}$.}
	\label{alg:slicing}
}
\end{algorithm}

\begin{remark}
When $\mathbf{F}$ is the identity map (as in many immersed methods), \emph{knot surfaces} are planes, and the slicing is a simple task.
Instead, when  $\mathbf{F}$ is not trivial, the slicing involve several geometric operations.
In this work, all these geometrical operations, e.g.\ 
computations of surface to surface intersections, points inversion, curves approximation, etc.,
are performed using the geometric kernels IRIT \cite{IRIT} and Open CASCADE \cite{OpenCASCADE}.
In both tools, and probably in most of the geometric kernels available,
these operations are executed with limited precisions (around $10^{-8}$ at best),
that are far from the maximum precision of 64-bits floating point numbers (around $10^{-16}$).

Even if these precisions are small enough for most of the industrial applications they target,
they impact the asymptotic behavior of our methods, for fine enough meshes. Some of the results presented in Section \ref{sec:experiments} reflect this issue.
\end{remark}

\subsubsection{Local re-parameterizations of cut B\'eziers}  \label{sec:local-rep-bezier}
After the slicing procedure described above we obtained a V-rep associated to the active part $\mathsf{Q}$ of every cut element.
In order to be able to compute numerical integrals over the cut elements, we re-parameterize each of those
V-reps with the purpose of building suitable integration formulae.

The element re-parameterization we propose in this work consists in creating a high-order hybrid polynomial mesh $\mathsf{Q}_h$, composed of tetrahedra and / or hexahedra. 
The created mesh $\mathsf{Q}_h$ is the union of $n_{\mathcal{L}}$ Lagrangian elements
$\mathsf{Q}_h=\cup^{n_{\mathcal{L}}}_{i=1} \mathcal{L}_i$, where $\mathcal{L}_i\subset\mathbb{R}^3$ denotes a single Lagrangian element.
Roughly, each element $\mathcal{L}_i$ is defined by a set of basis functions (Lagrange polynomials)
and the nodes associated to those functions.
This operation is sketched in the first step of Figure \ref{fig:reparam:2}.
\begin{figure}
\centering
\includegraphics[width=0.95\textwidth]{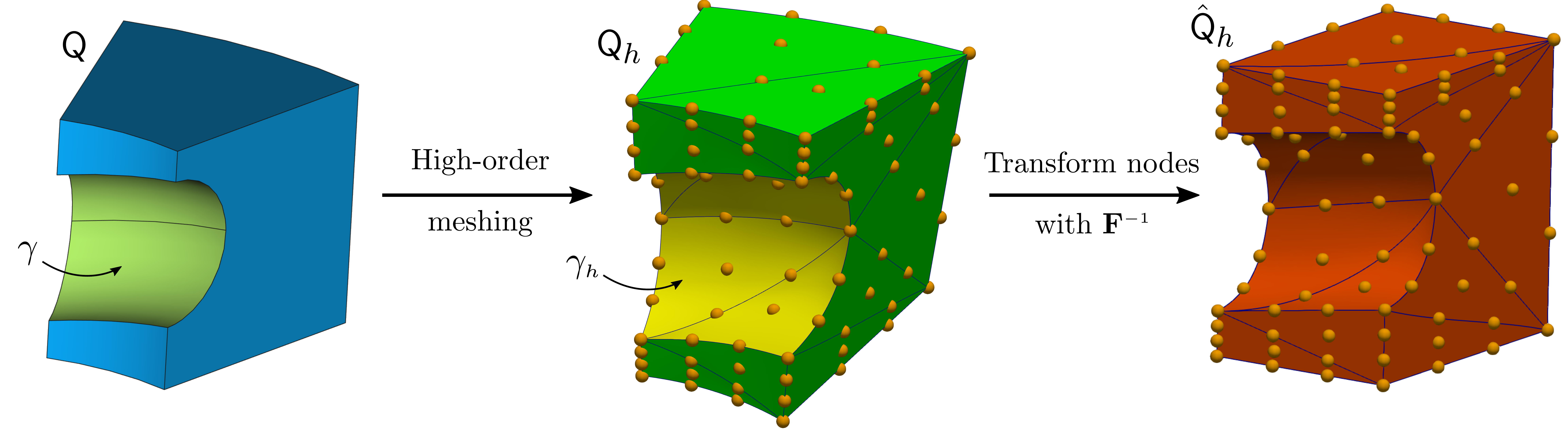}
\caption{Creation of a re-parameterization $\mathsf{Q}_h$ from the V-rep of a cut B\'ezier element $\mathsf{Q}$. In the right part of
the figure, the mesh $\mathsf{Q}_h$, contained in the physical domain, is transformed into
a mesh $\hat{\mathsf{Q}}_h$ in the parametric domain. The orange points represent the nodes of the tetrahedral meshes.}
\label{fig:reparam:2}
\end{figure}

In this work, the mesh $\mathsf{Q}_h$ is created using Gmsh software \cite{Gmsh}.
Gmsh is able to create high-order (up to degree 10) tetrahedral meshes from geometries defined by means of B-reps.
In addition, Gmsh can use internally Open CASCADE as geometric kernel for the treatment of B-rep geometries,
what makes the integration of both tools seamless.
Thus, in our implementation, we directly use the B-rep of every cut B\'ezier element created with  Open CASCADE and
generate a mesh $\mathsf{Q}_h$ with Gmsh.

In order to construct a 3D mesh, Gmsh first constructs meshes with linear elements
for the edges and faces of the B-rep. The nodes of the 1D and 2D elements
are created by sampling points on the surfaces and edges of the B-rep model, thus, it is guaranteed that
those points lay on the B-rep surfaces.
Afterwards, the mesh that fills the interior of the B-rep is created.

In a second step, Gmsh performs a degree elevation of the tetrahedral mesh elements by introducing new nodes.
As before, the new nodes created in this process, belonging to external faces of the mesh,
are sampled in the spline surfaces of the B-rep.
Hence, the boundary $\gamma$ is approximated with a high-order polynomial representation $\gamma_h$
(yellow surfaces in Figure \ref{fig:reparam:2}). Finally, as it is described in \cite{johnen_geometrical_2013,toulorge_robust_2013}, the mesh quality is optimized, \review{with an algorithm that guarantees} that all vertices lying on the B-rep surfaces will still belong to those surfaces. 
As our mesh is only used to place quadrature points, our lonely quality criteria is the absence of negative Jacobians: in practice, these relaxed quality requirements makes the task of creating high-order meshes much simpler that in finite element methods.
In addition, for the sake of computational efficiency, our aim is to create the coarsest possible valid mesh. 

The approximation $\mathsf{Q}_h$ obtained with this approach
complies with all the requirements stated in Assumption \ref{ass:approx-domain}, upon
which are obtained the estimate results presented in Section \ref{sec:integr_error_analysis}. Indeed, by construction, the boundary mesh $\gamma_h = \partial \mathsf{Q}_h$ is a high-order approximation of $\gamma = \partial \mathsf{Q}$. It may happen though that $\gamma_h \not\subset Q$: in this case, a dyadic subdivision of $\mathsf{Q}$ is performed and the meshing process is restarted on each concerned sub-element.

In a standard way, for calculating the operators involved in the problem \eqref{eq:varform-discr2}, their integrals are transformed from the physical to the parametric domain as:
\begin{equation}
\int_{\Omega} \xi = \int_{\hat \Omega} \xi\circ\bF\,\review{\det(\hat D\bF)}\,.
\end{equation}
Thus, once a valid mesh is created with Gmsh, we have a hybrid tetrahedral / hexahedral mesh of degree $r$, but defined in the physical space
(green mesh in Figure \ref{fig:reparam:2}). In order to use that mesh for building quadrature rules, we need to pull it back to the parametric domain \review{$\hat T$}.

Through a point inversion algorithm, all nodes of the mesh are pulled back to the parametric domain \review{$\hat T$} and the corresponding mesh $\hat{\mathsf{Q}}_h$, pre-image of $\mathsf{Q}_h$,   is generated. If needed, the Gmsh optimization referenced above is also applied to $\hat{\mathsf{Q}}_h$ in order to avoid negative Jacobians.

Once the mesh $\hat{\mathsf{Q}}_h$ is created,  in order to compute integrals over a single cut B\'ezier element we 
proceed as follows: we place a suitable quadrature rule (Gauss Legendre rule for hexahedra and collapsed Gauss Legendre for tetrahedra \cite{lyness_survey_1994}) defined over the reference element (e.g., the reference tetrahedron in Figure \ref{fig:reparam:3});
the quadrature coordinates and weights are transformed by pushing them forward
to all the elements in $\hat{\mathsf{Q}}_h$ in  \review{$\hat T$} (right part of Figure \ref{fig:reparam:3});
using those transformed quadratures, 
the contributions of all the elements are added to the global result. 
\begin{figure}
\centering
\includegraphics[width=0.95\textwidth]{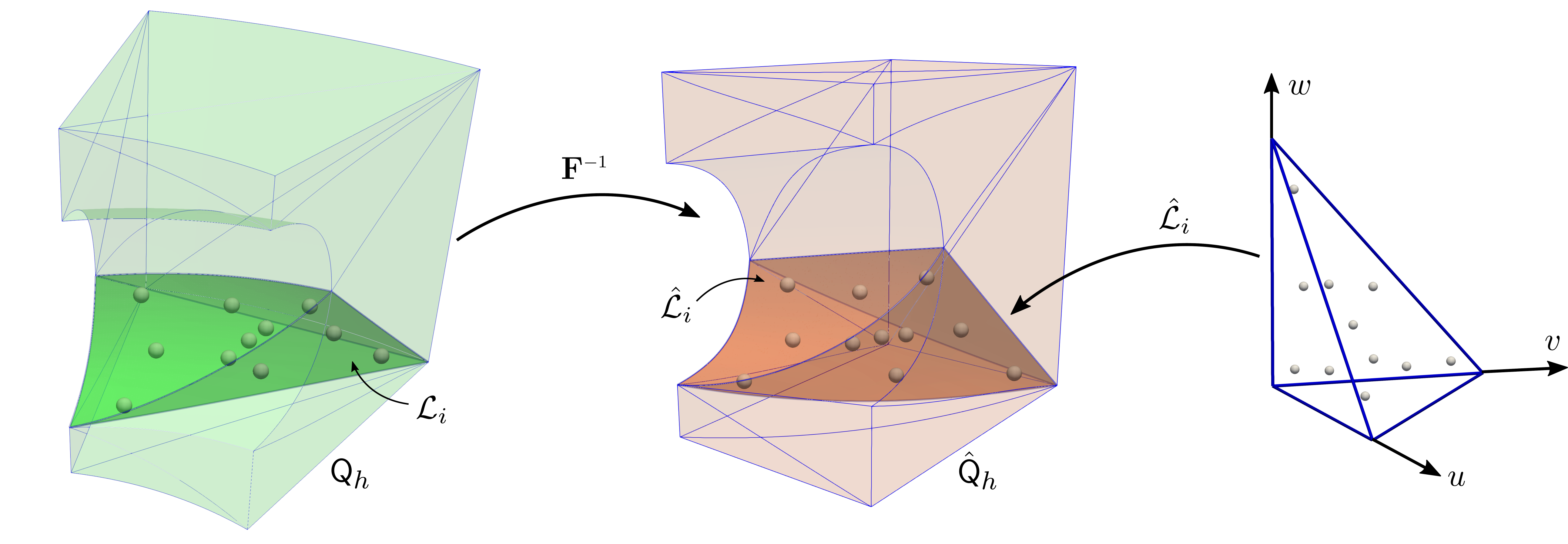}
\caption{Creation of quadrature rules by setting points in the reference element (right tetrahedron) and pushing them forward to the parametric domain (in red).
The quantities to integrate are pulled-back from the physical domain (in green) to the parametric one by applying the inverse transformation $\bF^{-1}$.
The images of the quadrature points in the different domains are presented as white points.}
\label{fig:reparam:3}
\end{figure}

\begin{remark}
In order to avoid the creation of the mesh $\hat{\mathsf{Q}}_h$ from $\mathsf{Q}_h$, an alternative
procedure would be to set the quadrature points in the elements of the physical mesh $\mathsf{Q}_h$ and
then transform their coordinates and weights back to the parametric domain $\hat T$ by applying the inverse transformation $\bF^{-1}$.
This procedure presents the advantage of bypassing the creation of the mesh $\hat{\mathsf{Q}}_h$, however,
it is dependent of the integration formula chosen and must be repeated for different quadrature rules.
\end{remark}
\begin{remark}
 The choice of the quadrature rule we made is likely not optimal, and, more generally, the optimisation of the quadrature is not the main focus of this contribution and deserves further studies. 
\end{remark}
\begin{remark}
Even if the procedure presented above is applicable also to the 2D domains, in that case is far from optimal.
The first operation, the \emph{slicing}, is common for both 2D and 3D cases, however,
the way in which local re-parameterization of the trimmed elements is created can be greatly simplified in the 2D case.
See, for instance, \cite{Marussig2018, kudela_efficient_2015, benson15}.
The procedure used in this work for the 2D experiments presented in Section \ref{sec:experiments}
is based on the use of patterns as described in \cite{Marussig17,schmidt_isogeometric_2012,kim_isogeometric_2010}.
When a cut B\'ezier element does not correspond to any pattern, we use the methodology presented above.
\end{remark}

\subsection{Linear system preconditioning} 
\label{sec:conditioning}
The error estimates obtained in Section \ref{sec:integr_error_analysis} state
that, under certain assumptions, guaranteed by the constructions presented in Section \ref{sec:local-rep},
the accuracy and optimality of the isogeometric method are preserved.

Nevertheless, due to the fact that the support of some basis functions is cut, the linear system
can be potentially ill-conditioned. This is due to the fact that by reducing the active support
of some basis functions, we are also reducing the minimum eigenvalues of the linear system matrix, while the maximum eigenvalue is
not affected.

A detailed study of conditioning problems that derive from
the use of trimmed domains is beyond the scope of this work, but we include here a brief discussion of the preconditioner
used in the numerical examples presented in Section \ref{sec:experiments}.

In this work we use a diagonal scaling preconditioner.  
Thus, the original linear system can be written as $\bm{\mathsf{A}}\bm{\mathsf{x}}=\bm{\mathsf{b}}$,
where $\bm{\mathsf{A}}\in\mathbb{R}^{n\times n}$ is the stiffness matrix, $\bm{\mathsf{x}},\,\bm{\mathsf{b}}\in\mathbb{R}^{n}$ are the solution and right-hand-side vectors, respectively, and $n$ is the system size.
By symmetrically preconditioning the system it becomes:
\begin{subequations}
\begin{align}
\bm{\mathsf{D}}\bm{\mathsf{A}}\bm{\mathsf{D}}\bm{\mathsf{y}}&=\bm{\mathsf{D}}\bm{\mathsf{b}}\,,\\
\bm{\mathsf{x}}&=\bm{\mathsf{D}}\bm{\mathsf{y}}\,,
\end{align}
\end{subequations}
where $\bm{\mathsf{D}}$ is the diagonal scaling preconditioner defined as
\begin{align}
\bm{\mathsf{D}} &= \text{diag}\left( 1/\sqrt{A_{1,1}}\,,\,1/\sqrt{A_{2,2}}\,,\dots,\,1/\sqrt{A_{n,n}} \right)\,,
\end{align}
where $A_{i,i}$, with $i=1,\dots,n$, are the diagonal entries of the matrix $\bm{\mathsf{A}}$.
By applying this algebraic preconditioner, all the basis functions are renormalized respect to their contribution to the matrix $\bm{\mathsf{A}}$. 

This preconditioner's effectiveness is illustrated, for a series of elliptic problems,
in the numerical experiments gathered in Section \ref{sec:experiments}. Although in the context of the Finite Cell Method \cite{de_prenter_condition_2017}, this simple choice seems  insufficient to restore a stable behavior of the condition number, i.e., a  condition number independent of the way elements are cut,   in our numerical tests  we consistently experience a very good behavior of the preconditioned system. We believe this is due to the use, in our tests, of maximum continuity splines and we  defer the reader to \cite{buffa_minimal_2019} for further discussions. 

\review{
Possible alternatives to the diagonal scaling preconditioning are the ghost penalty techniques \cite{burman_ghost_2010,burman_cutfem:_2015};
the use of additive Schwarz preconditioners as proposed in \cite{de_prenter_preconditioning_2019,jomo_robust_2019} for finite cell and isogeometric discretizations, and extended in \cite{de_prenter_scalable_2019} to case of immersed IGA combined with multigrid methods; or the stable removal of basis functions, as proposed in \cite{Elfverson2018} in the context of CutIGA methods.}

\section{Numerical experiments} \label{sec:experiments}
In this section we perform a series of numerical experiments
that aim at illustrating the methodology presented in previous sections.

As a first example,
we present in Section \ref{sec:poisson_2D} a Poisson problem in a 2D circular domain
which will help us in supporting the theoretical results presented in Section \ref{sec:theory}.
Secondly, the well known plate with hole
test for 2D linear elasticity is presented (Section \ref{sec:plate_with_hole}).
A Poisson problem is used again in Section \ref{sec:poisson_3D} to validate our theoretical findings, 
but this time on a three-dimensional computational domain: a sphere, which is trimmed from a cube. 
Finally, the presented methodology is applied to study the elastic
behavior of a 3D mold with a trimmed interior cooling channel, which has a complex geometry.

The results shown in this section have been obtained using
an implementation of the methodology proposed above in an in-house code built
upon the isogeometric analysis library Igatools \cite{Igatools},
the solid modeling kernels IRIT \cite{IRIT} and OpenCASCADE \cite{OpenCASCADE} and 
the 3D mesh generator Gmsh \cite{Gmsh}.

\subsection{Poisson equation in a 2D domain}\label{sec:poisson_2D}
In this Section we study a Poisson problem in the interior of a circular domain $\Omega$:
\begin{subequations}
\begin{alignat}{2}
\Delta u &= f\quad&&\text{in }\Omega\,,\\
\nabla u\cdot\bm{n} &= g_n\quad&&\text{on }\partial\Omega\,.
\end{alignat}
\end{subequations}
The solution to the problem is
\begin{align}
u = \sin\left(\frac{2\pi}{L}x\right)\sin\left(\frac{2\pi}{L}y\right)\,,
\end{align}
and the associated source and Neumann terms are:
\begin{subequations}
\begin{align}
f  &= -\frac{8\pi^2}{L^2}\sin\left(\frac{2\pi}{L}x\right)\sin\left(\frac{2\pi}{L}y\right)\,,\\
g_n  &= \frac{2\pi}{L}\left( x\cos\left(\frac{2\pi}{L}x\right)\sin\left(\frac{2\pi}{L}y\right) + y\sin\left(\frac{2\pi}{L}x\right)\cos\left(\frac{2\pi}{L}y\right)\right)\,.
\end{align}
\end{subequations}
No Dirichlet condition is imposed, alternatively we impose weakly the condition
\begin{align}
\int_\Omega u = \int_\Omega\sin\left(\frac{2\pi}{L}x\right)\sin\left(\frac{2\pi}{L}y\right)=0\,,
\end{align}
in order to remove the constant part of $u$.

\subsubsection{Poisson equation in a 2D non-distorted domain}
The computational domain is built as the boolean intersection of a $L$-length square domain
and a circle, with radius $R$, centered at the origin (c.f.\ Figure \ref{fig:poisson_2D_sketch}). The squared domain carries a Cartesian mesh. 
As a matter of example, in Figure \ref{fig:poisson_2D_sketch_reparam_2} the re-parameterization of trimmed elements (in red) is shown together with the non-trimmed ones for a domain with $8\times8$ elements.
\begin{figure}[t!]
 \centering
 \subfigure[Sketch of the Poisson 2D problem.\label{fig:poisson_2D_sketch}]
 {\ifdraftversion\else\includegraphics{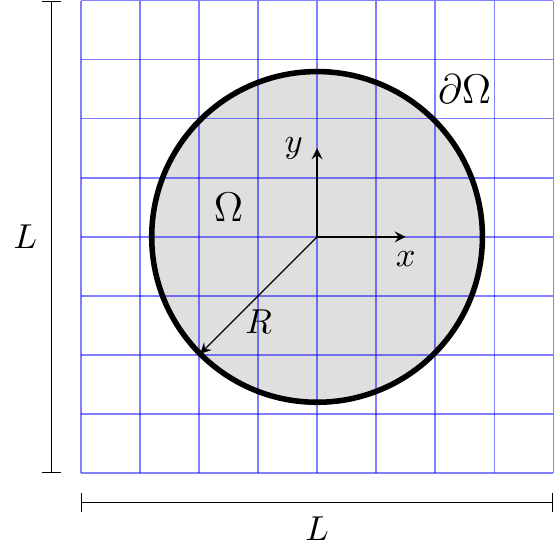}\fi}\hfill
 \subfigure[Physical domain $\Omega$. In red, re-parameterization of trimmed elements, in blue, non-trimmed ones.\label{fig:poisson_2D_sketch_reparam_2}]
 {\ifdraftversion\else\includegraphics{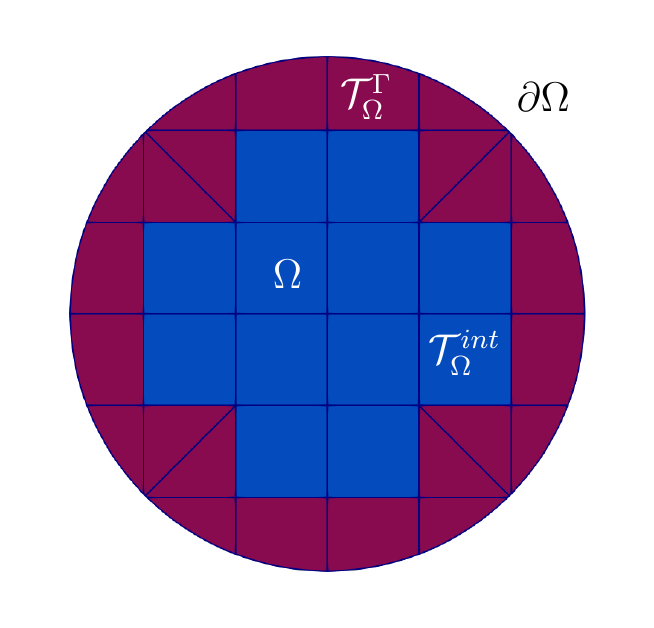}\fi}
 \caption{Sketch of the Poisson 2D problem (left) and illustration of the trimmed and non-trimmed elements (right) for a domain with $8\times8$ elements and $R=1$ and $L=2\,R/0.7$.}
\label{fig:poisson_2D_sketch_reparam}
\end{figure}
For this re-parameterization, as well as for all the numerical results gathered in this section, $R=1$ and $L=2\,R/0.7$ were used. As the image of each knot line is straight, 
the re-parameterization of the trimmed elements is performed  using a numerical precision $10^{-15}$, it is of the same degree as the one used for solution, and is the coarsest possible.

As a first result, the solution error in the $L^2$ and $H^1$ norms for different mesh sizes and degrees
is shown in Figures \ref{fig:poisson_2D_l2_h1_area_perimeter_l2} and \ref{fig:poisson_2D_l2_h1_area_perimeter_h1}, respectively,  with the choice $p=r$, in agreement to Theoreom \ref{th:approx}. For $p=4,5$ one quadrature point per direction is added in each element of the re-parameterization, and for $p=6$, two points per direction are added in order to recover the expected convergence rates.

Moreover, \ref{fig:poisson_2D_l2_h1_area_perimeter_area} and \ref{fig:poisson_2D_l2_h1_area_perimeter_perimeter}  show the error on the computation of area and perimeter which are bounded from below by the chosen geometric precision. 
Note that the limited accuracy of the re-parameterization for $p=6$ and $h\leq L/64$ (c.f.\ Figure \ref{fig:poisson_2D_l2_h1_area_perimeter_perimeter})
has an effect in the computed $L^2$ error norm for $p=6$ and $h=L/128$. 
Finally, from \ref{fig:poisson_2D_l2_h1_area_perimeter_area} and \ref{fig:poisson_2D_l2_h1_area_perimeter_perimeter},  the area and perimeter errors converge as $h^{p+1}$ and $h^{p+2}$ for odd and even degrees,
respectively.  This fact is due to the use of equidistant interpolation points for the approximation of the boundary, and the known fact that quadrature formulae on equidistant odd points are always even degree.
\begin{figure}
 \centering
 \subfigure[Solution error in the $L^2$ norm.\label{fig:poisson_2D_l2_h1_area_perimeter_l2}]
 {\ifdraftversion\else\includegraphics{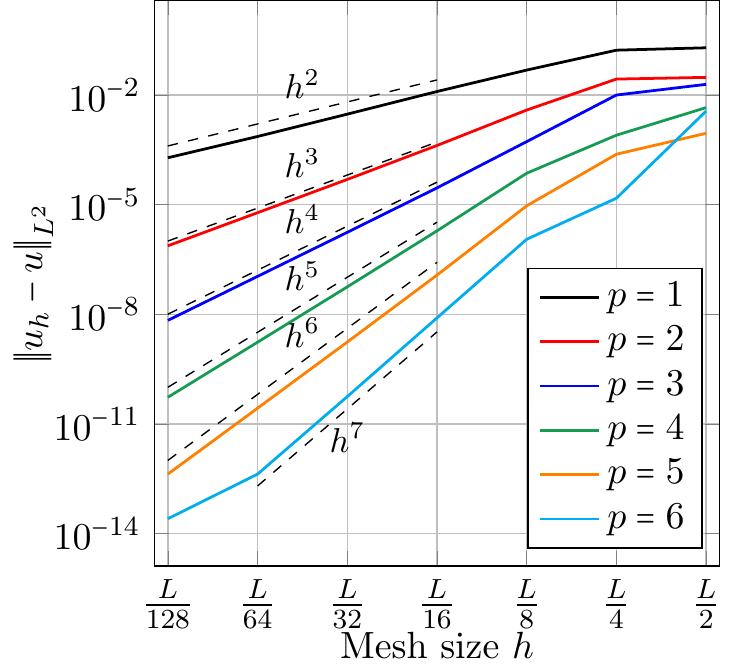}\fi}\hfill
 \subfigure[Solution error in the $H^1$ norm.\label{fig:poisson_2D_l2_h1_area_perimeter_h1}]
 {\ifdraftversion\else\includegraphics{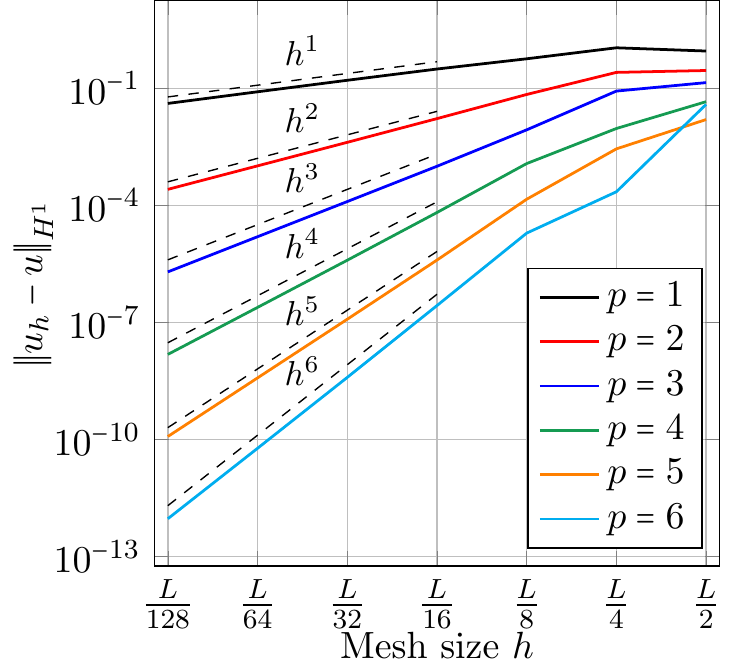}\fi}\\
 \subfigure[Area error.\label{fig:poisson_2D_l2_h1_area_perimeter_area}]
 {\ifdraftversion\else\includegraphics{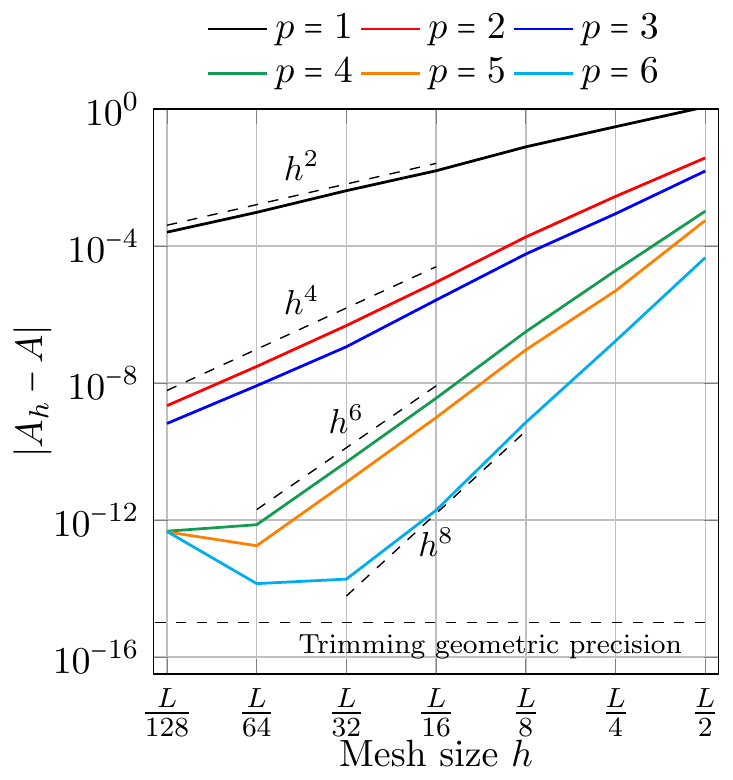}\fi}
 \subfigure[Perimeter error.\label{fig:poisson_2D_l2_h1_area_perimeter_perimeter}]
 {\ifdraftversion\else\includegraphics{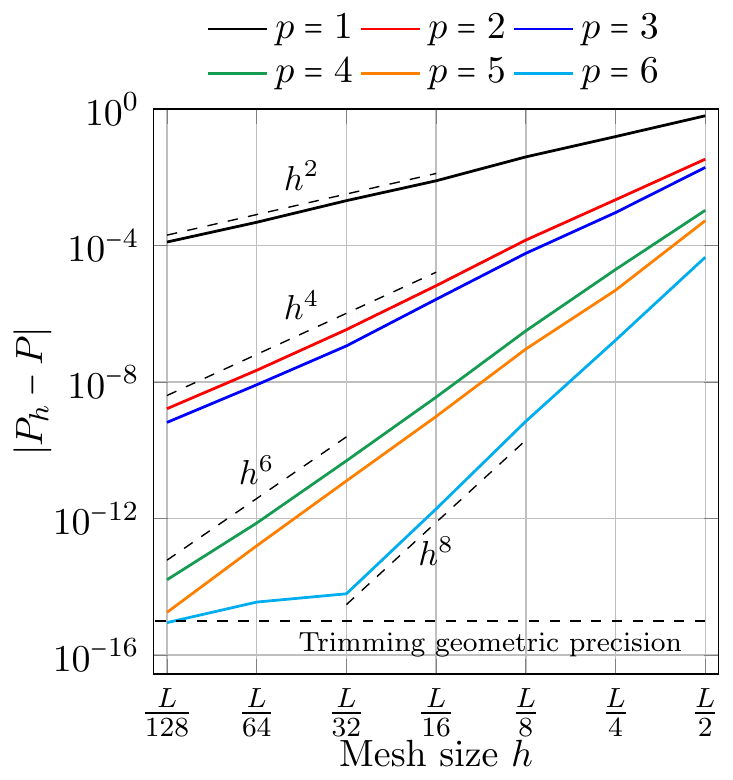}\fi}
 \caption{Poisson 2D problem: error of the solution $u_h$ in the $L^2$ and $H^1$ norms, and absolute error of area and perimeter computation, for different
 degrees.
 In all cases the re-parameterization of the trimmed elements is the coarsest possible and was performed with the same degree $p$ used for the discretization of the solution.}
 \label{fig:poisson_2D_l2_h1_area_perimeter}
\end{figure}

In order to illustrate the necessity of high-order re-parameterizations for the trimmed elements,
in Figure \ref{fig:poisson_2D_h1_different_pts} we show the effect
of using re-parameterization degrees $p_t$ lower than the discretization degree $p$.
In this figure we split the contribution to the $H^1$ error norm in two parts:
the contribution of the trimmed elements $\T^{\Gamma}_\Omega$ (solid lines) and
of the non-trimmed ones $\T^{int}_\Omega$ (dashed lines).

As it can be seen, when $p_t<p$ and $h$ becomes small
the consistency error starts to be dominant respect to the discretization error
and consequently the error optimality is lost.  For higher $p_t$, this phenomenon alleviates and the sub-optimality shows up on finer and finer meshes.
\begin{figure}
 \centering
 \subfigure[$p=3$.\label{fig:poisson_2D_h1_different_pts_3}]
 {\ifdraftversion\else\includegraphics{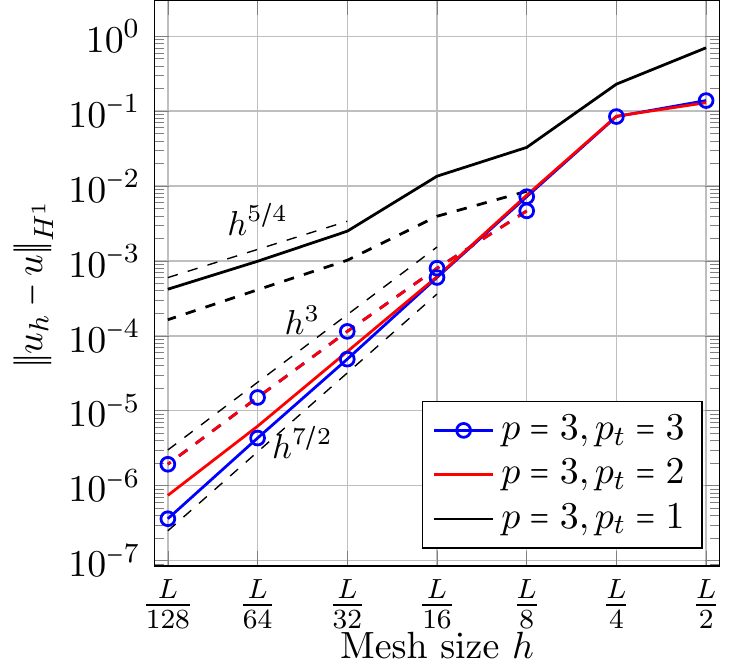}\fi}
 \subfigure[$p=5$.\label{fig:poisson_2D_h1_different_pts_5}]
 {\ifdraftversion\else\includegraphics{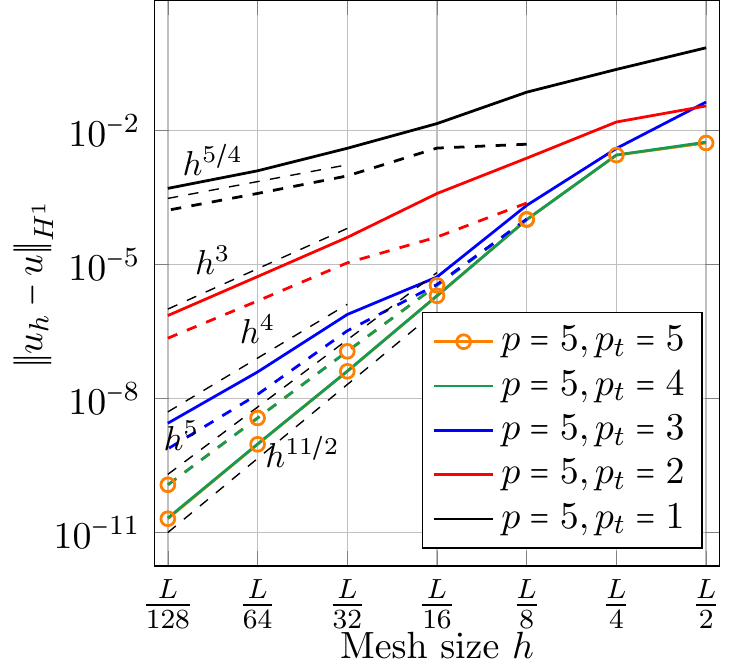}\fi}
 \caption{Poisson 2D problem: error in the $H^1$ norm of the solution $u_h$ for different discretization degrees $p$ using
 degree $p_t$ for the re-parameterization of the trimmed elements.
 The error norm is split in the contribution of the trimmed elements $\T^{\Gamma}_\Omega$ (solid lines) and the non-trimmed ones $\T^{int}_\Omega$ (dashed lines).}
 \label{fig:poisson_2D_h1_different_pts}
\end{figure}
\begin{figure}
 \centering
 \subfigure[$p=3$.]
 {\ifdraftversion\else\includegraphics{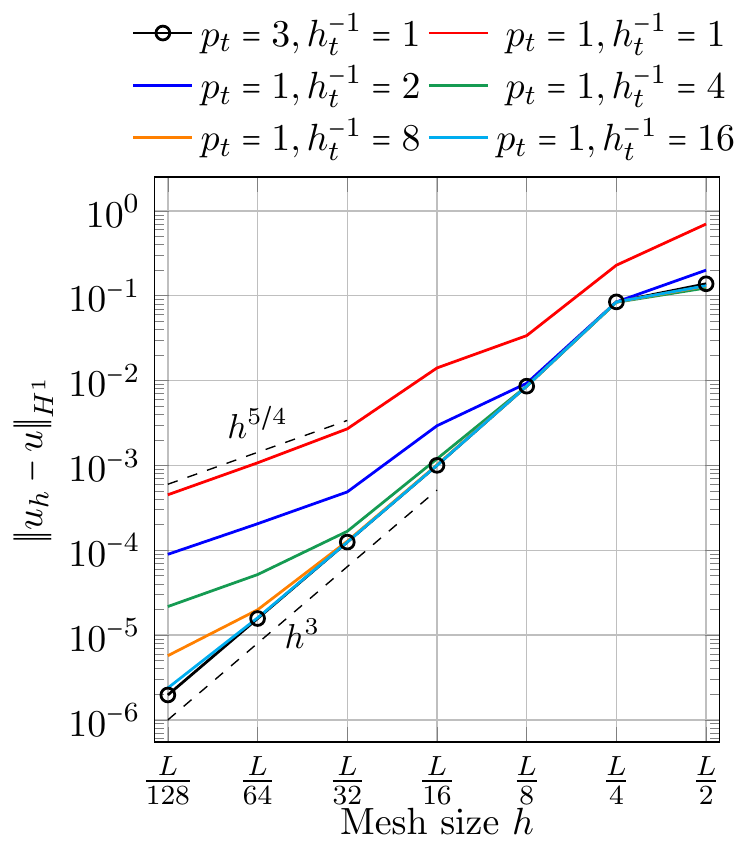}\fi}\hfill
 \subfigure[$p=4$.]
 {\ifdraftversion\else\includegraphics{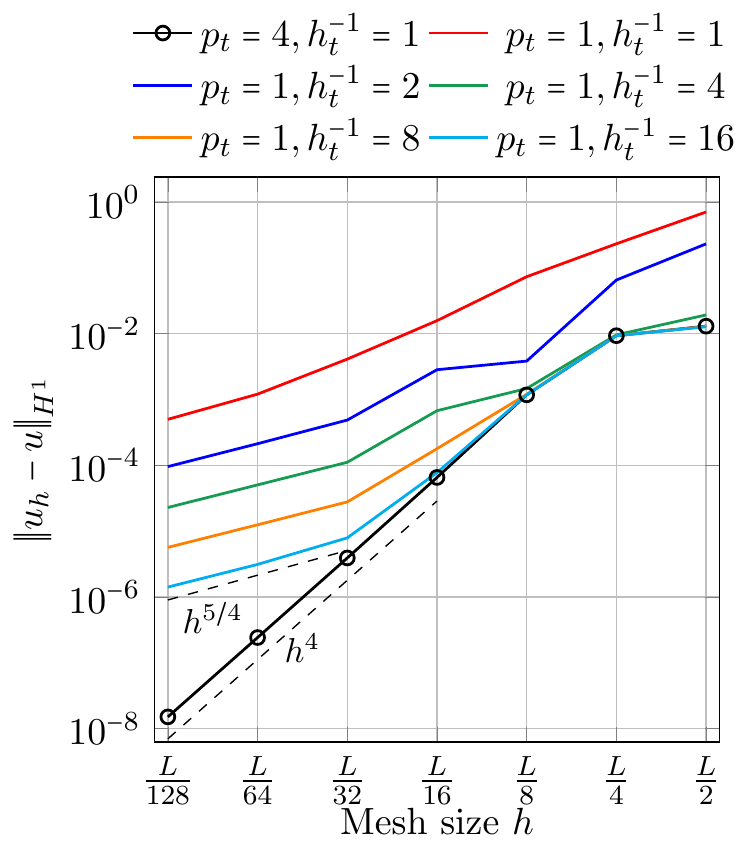}\fi}
 \caption{Poisson 2D problem: error in the $H^1$ norm
 for different solution degrees $p$ using a low-order re-parameterization ($p_t=1$) and an increasing number of re-parameterization tiles ($h^{-1}_t$) for every trimmed element.}
 \label{fig:poisson_2D_h1_different_hts}
\end{figure}

On the other hand, in Figure \ref{fig:poisson_2D_h1_different_hts}, despite the solution degree $p$,
the trimmed elements are re-parameterized with a low-order approximation ($p_t=1$).
 Nevertheless, instead of creating the coarsest possible re-para\-mete\-ri\-za\-tion, the trimmed elements
are approximated using more tiles. $h_t$ is the size of the tiles and
$h^{-1}_t$ measures the number of tiles used along each side of every trimmed element.
The use of a higher number of tiles has the effect of reducing the consistency error, thus its sub-optimality affects the total error for finer meshes. 

As anticipated in Remark \ref{rem:subopt}, the convergence rates we see in Figures \ref{fig:poisson_2D_h1_different_pts} and \ref{fig:poisson_2D_h1_different_hts} are slightly better than the ones proved true in Theorem \ref{th:approx}. 

Finally, in Figure \ref{fig:poisson_2D_conditioning} the stiffness matrices' condition numbers are reported
for the preconditioned and non-preconditioned cases. As it can be seen, the non-preconditioned case
(Figure \ref{fig:poisson_2D_conditioning_without}) presents very high condition numbers for all degrees.
This is controlled, as it can be appreciated in Figure \ref{fig:poisson_2D_conditioning_with}, using
the diagonal preconditioning described in Section \ref{sec:conditioning}. Additionally, the expected
$h^{-2}$ asymptotic behavior of the conditioning is recovered in the latter case.
\begin{figure}
 \subfigure[Without preconditioning.\label{fig:poisson_2D_conditioning_without}]
 {\ifdraftversion\else\includegraphics{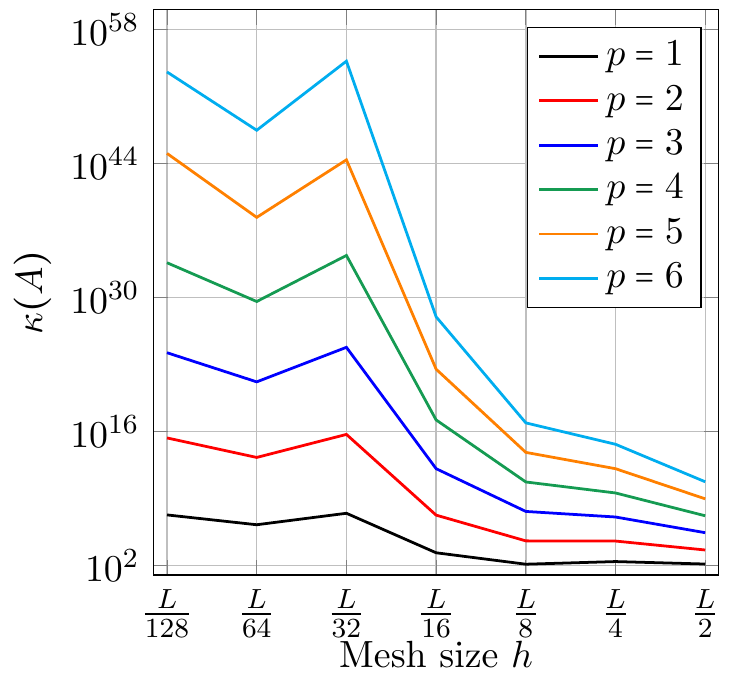}\fi}\hfill
 \subfigure[With diagonal scaling preconditioning.\label{fig:poisson_2D_conditioning_with}]
 {\ifdraftversion\else\includegraphics{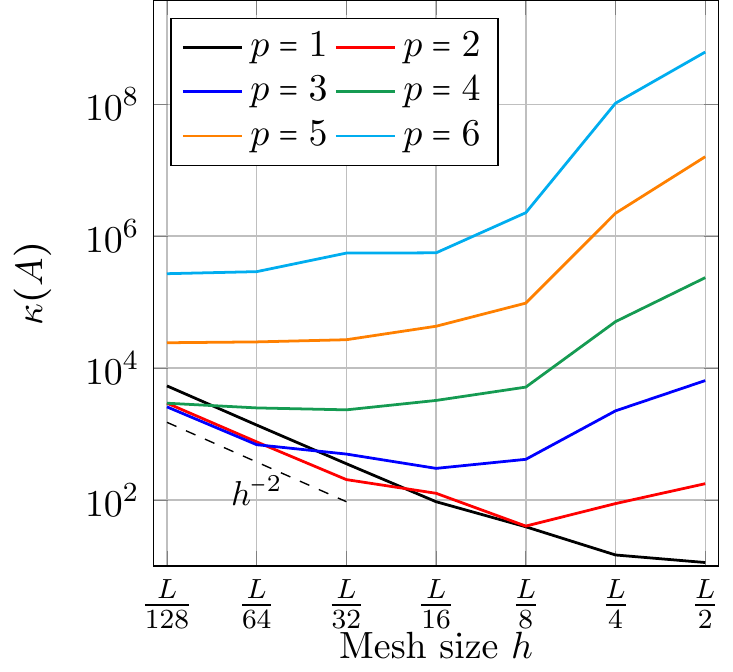}\fi}
 \caption{Poisson 2D problem: stiffness matrix conditioning with and without diagonal scaling preconditioning.}
 \label{fig:poisson_2D_conditioning}
\end{figure}

\subsubsection{Poisson equation in a 2D distorted domain}\label{sec:poisson_2D_distorted}
The same Poisson problem is solved but using a different parameterization of the computational domain.
The computational domain is still the interior of a circle, however the parameterization does not longer
correspond to a Cartesian mesh, but to a distorted one.
In Figure \ref{fig:poisson_2D_distorted_mesh} the trimmed parametric domain
is shown in the left side and the physical distorted domain in the right side. As described in Section \ref{sec:local-rep},  this time the geometric operations are performed approximately with a numerical precision of $10^{-8}$ that is far from the maximum precision
of 64-bits floating point numbers used in the calculations (around $10^{-16}$). 
\begin{figure}
 \centering
 \subfigure[Parametric domain.]
 {\ifdraftversion\else\includegraphics{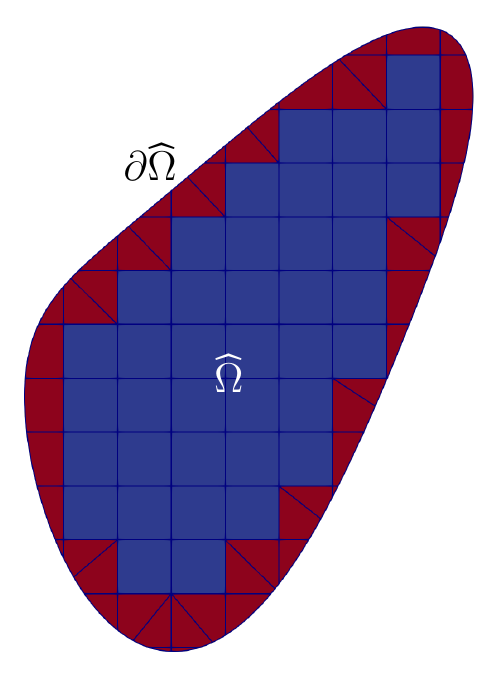}\fi}\hfill
 \subfigure[Physical domain.]
 {\ifdraftversion\else\includegraphics{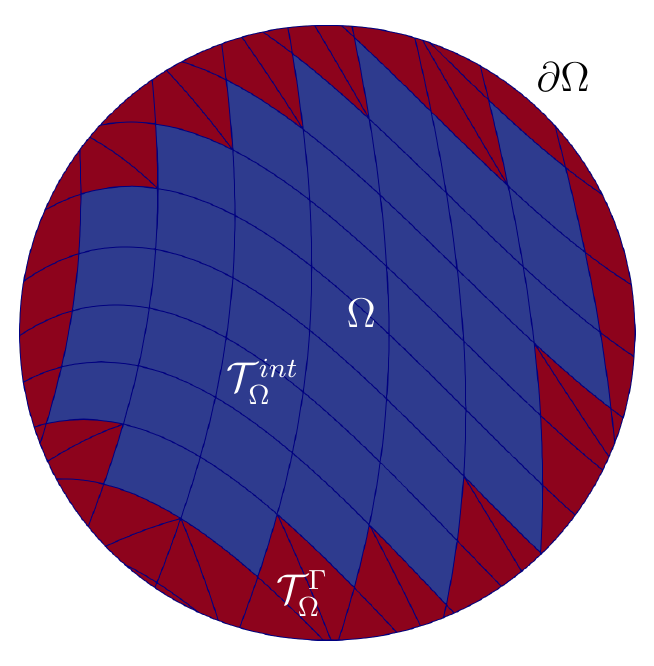}\fi}
 \caption{Poisson 2D problem in a distorted domain: parametric (left) and physical distorted (right) domains.
 In red, the re-parameterized trimmed elements $\T^{\Gamma}_\Omega$, in blue, the non-trimmed ones $\T^{int}_\Omega$.}
 \label{fig:poisson_2D_distorted_mesh}
\end{figure}
The aim of this section is to show the impact of this geometry approximation on the accuracy of the numerical method.

In Figure \ref{fig:poisson_2D_l2_h1_area_perimeter_distorted_occ} the solution error
in the $L^2$ and $H^1$ norms and the absolute error of the area
and external perimeter are shown for the distorted domain case. It can be seen that both the area and perimeter
errors, that have a noisy behavior due to the distortion of the parameterization,
are bounded from below by  the geometric precision.
On the other hand, the solution $L^2$ and $H^1$ errors (c.f.\ Figures
\ref{fig:poisson_2D_l2_h1_area_perimeter_distorted_l2_occ} and \ref{fig:poisson_2D_l2_h1_area_perimeter_distorted_h1_occ})
show optimal converges properties until they reach the geometric precision (which bounds them from below).
\begin{figure}
 \centering
 \subfigure[Solution error in the $L^2$ norm.\label{fig:poisson_2D_l2_h1_area_perimeter_distorted_l2_occ}]
 {\ifdraftversion\else\includegraphics{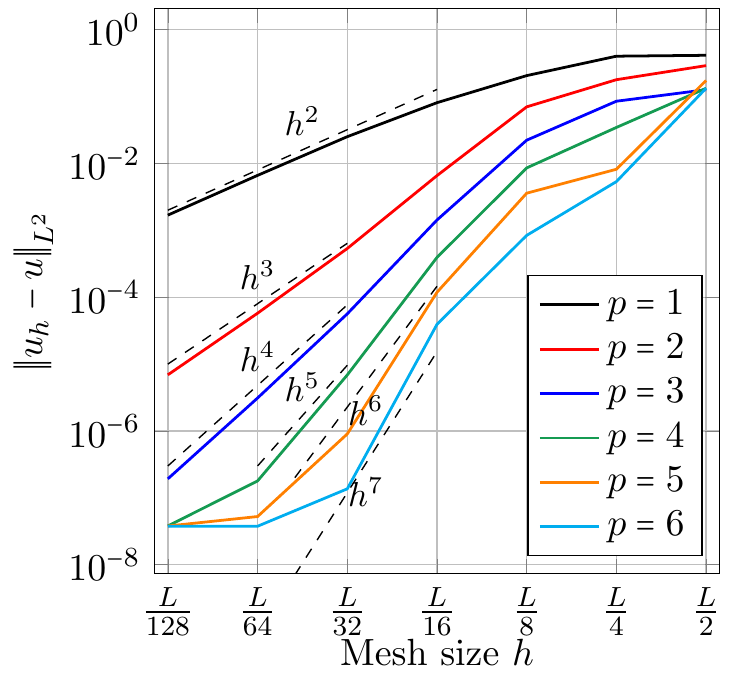}\fi}\hfill
 \subfigure[Solution error in the $H^1$ norm.\label{fig:poisson_2D_l2_h1_area_perimeter_distorted_h1_occ}]
 {\ifdraftversion\else\includegraphics{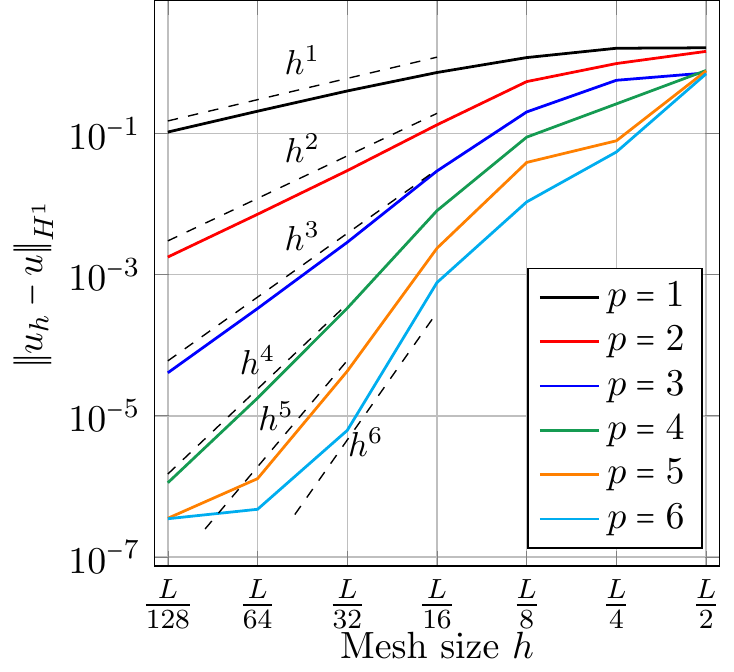}\fi}\\
 \subfigure[Area error.]
 {\ifdraftversion\else\includegraphics{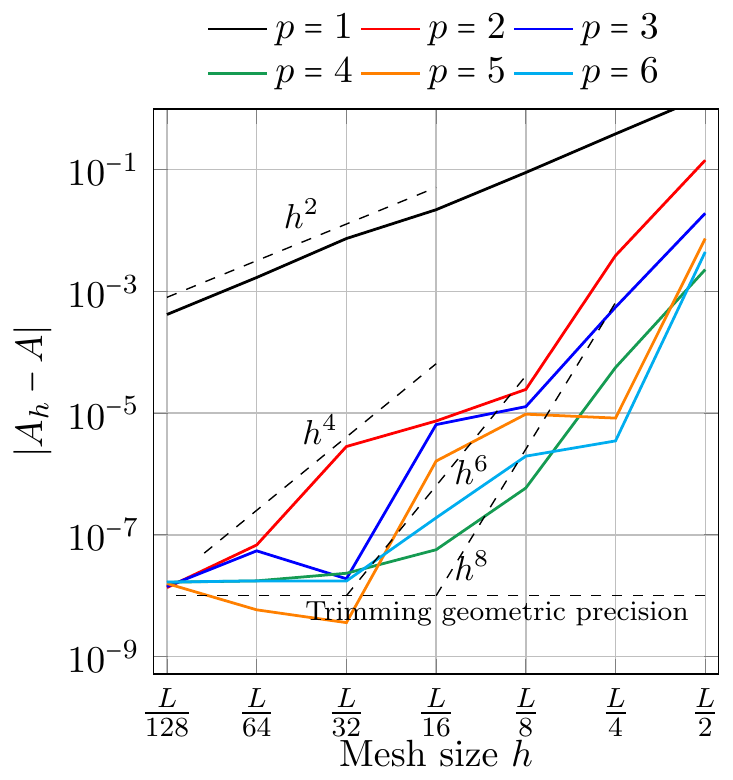}\fi}\hfill
 \subfigure[Perimeter error.]
 {\ifdraftversion\else\includegraphics{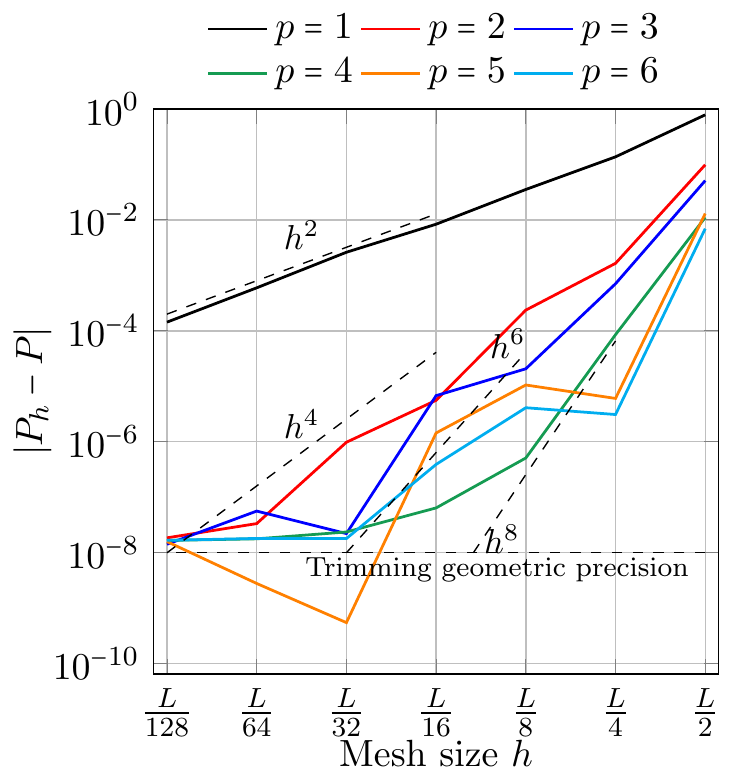}\fi}
 \caption{Poisson 2D problem in a distorted domain with a geometric precision $10^{-8}$:
 error of the solution $u_h$ in the $L^2$ and $H^1$ norms, and absolute error of area and perimeter computations with the re-parameterized domain, for different
 degrees.}
 \label{fig:poisson_2D_l2_h1_area_perimeter_distorted_occ}
\end{figure}

In order to confirm the interpretation of our numerical results, 
for this specific (simple) case, we were able to compute an approximation of the curve $\partial\hat\Omega$ with a precision under $10^{-12}$, by enforcing manually the precision of the geometric algorithms.
The obtained results are gathered in Figure \ref{fig:poisson_2D_l2_h1_area_perimeter_distorted}.
All the errors present the same behavior as in the previous case, but solution errors
are not longer bounded by the geometric precision.
\begin{figure}
 \centering
 \subfigure[Solution error in the $L^2$ norm.\label{fig:poisson_2D_l2_h1_area_perimeter_distorted_l2}]
 {\ifdraftversion\else\includegraphics{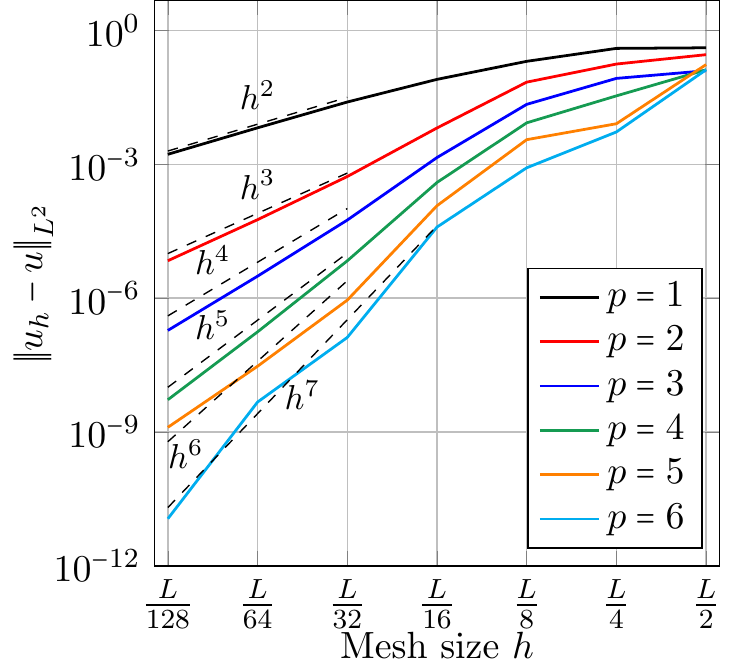}\fi}\hfill
 \subfigure[Solution error in the $H^1$ norm.\label{fig:poisson_2D_l2_h1_area_perimeter_distorted_h1}]
 {\ifdraftversion\else\includegraphics{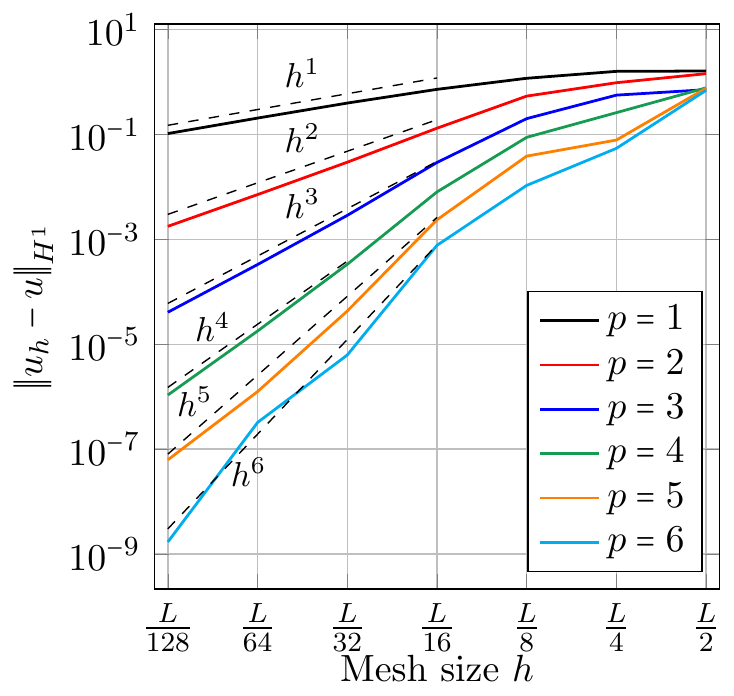}\fi}\\
 \subfigure[Area error.]
 {\ifdraftversion\else\includegraphics{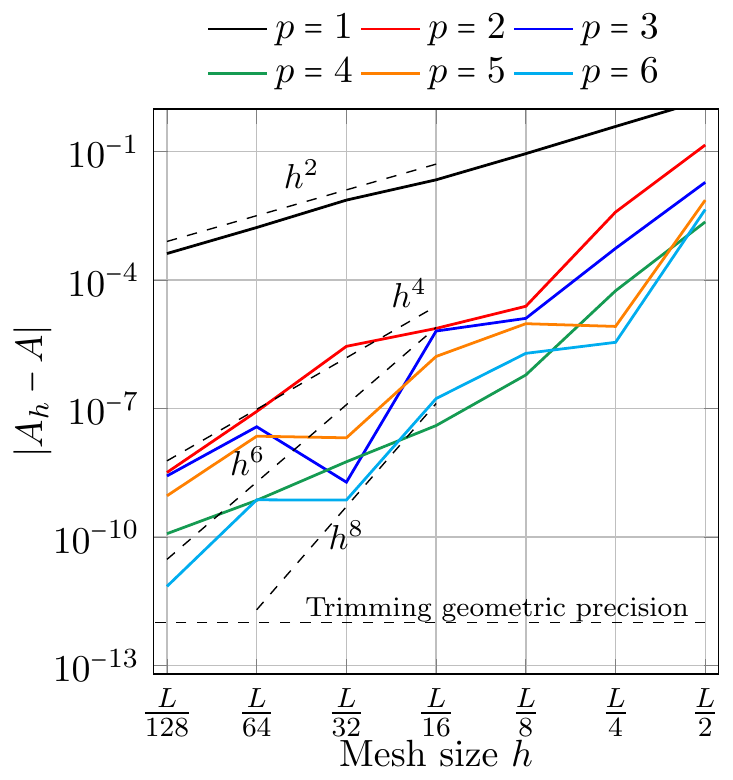}\fi}\hfill
 \subfigure[Perimeter error.]
 {\ifdraftversion\else\includegraphics{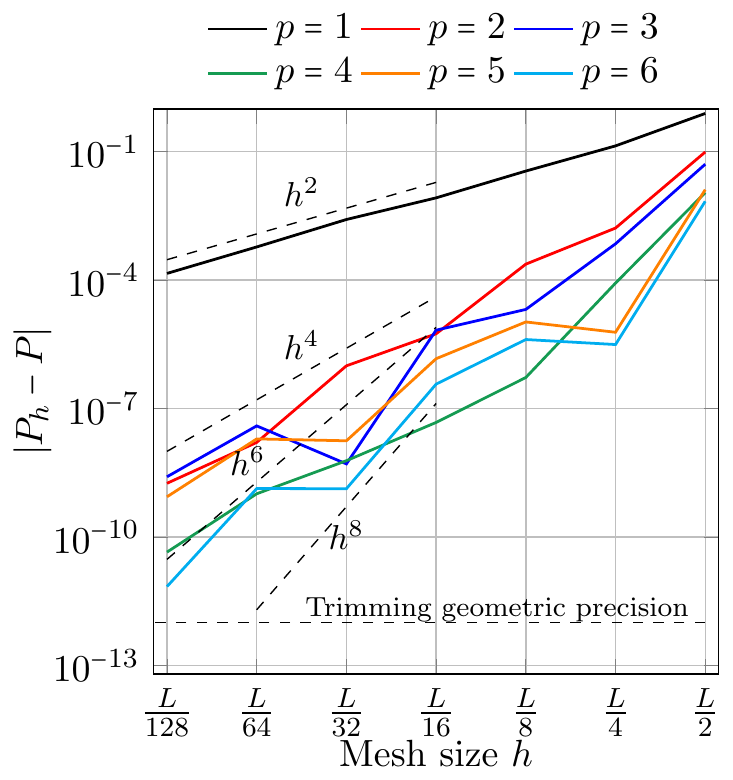}\fi}
 \caption{Poisson 2D problem in a distorted domain with a geometric precision $10^{-12}$:
  error of the solution $u_h$ in the $L^2$ and $H^1$ norms, and absolute error of area and perimeter computations with the re-parameterized domain, for different
 degrees.}
 \label{fig:poisson_2D_l2_h1_area_perimeter_distorted}
\end{figure}

\clearpage
\subsection{Infinite plate with circular hole in a 2D domain} \label{sec:plate_with_hole}
In this section the infinite plate with a circular hole under tension is studied,
a well-known case in the isogeometric literature
(e.g.\ \cite{hughes2005,bazilevs2006}).

Considering plain strain conditions for the problem, the analytic solution, expressed in polar coordinates, is:
\begin{subequations}
\begin{align}
    u_x(r,\,\theta) &= \frac{T_x R^2}{4\mu r} 
    \left[
      \left(2-2\nu\right)
        \left(\frac{r^2}{R^2} + 2\right)\cos\theta
        +\left(1-\frac{R^2}{r^2}\right)\cos3\theta
    \right]\,,\\
    u_y(r,\,\theta) &= \frac{T_x R^2}{4\mu r}
    \left[
      \left(4-4\nu-2\nu\frac{r^2}{R^2}\right)\sin\theta
        +\left(1-\frac{R^2}{r^2}\right)\sin3\theta
    \right]\,.
\end{align}
\end{subequations}
The involved quantities and the problem boundary conditions are described in Figure \ref{fig:plate_sketch}. Dirichlet boundary conditions are on the non-trimmed part of the boundary and are imposed strongly. 
\begin{figure}[h]
 \centering
 \subfigure[Description of the problem.\label{fig:plate_sketch}]
 {\ifdraftversion\else\includegraphics{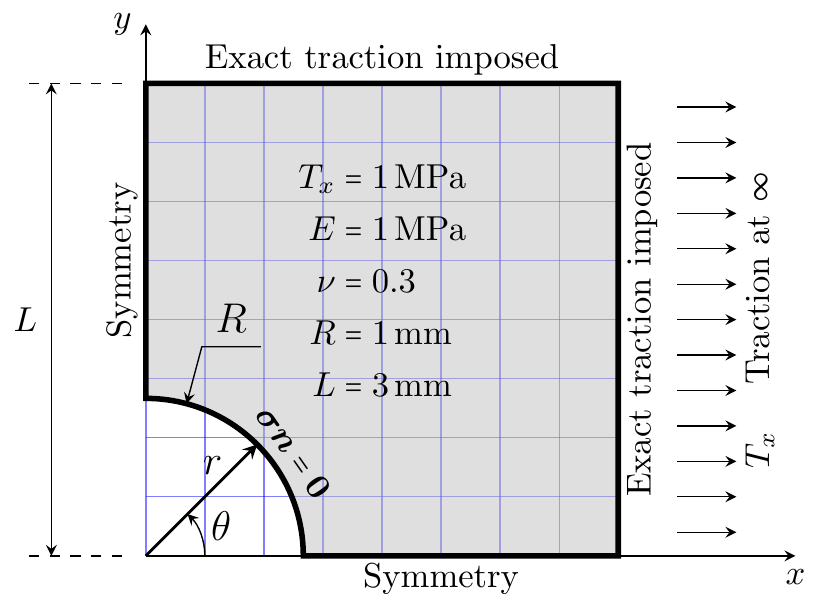}\fi}\hfill
 \subfigure[Physical domain $\Omega$. In red, re-parameterized trimmed elements, in blue, non-trimmed ones.\label{fig:poisson_2D_reparam}]
 {\ifdraftversion\else\includegraphics{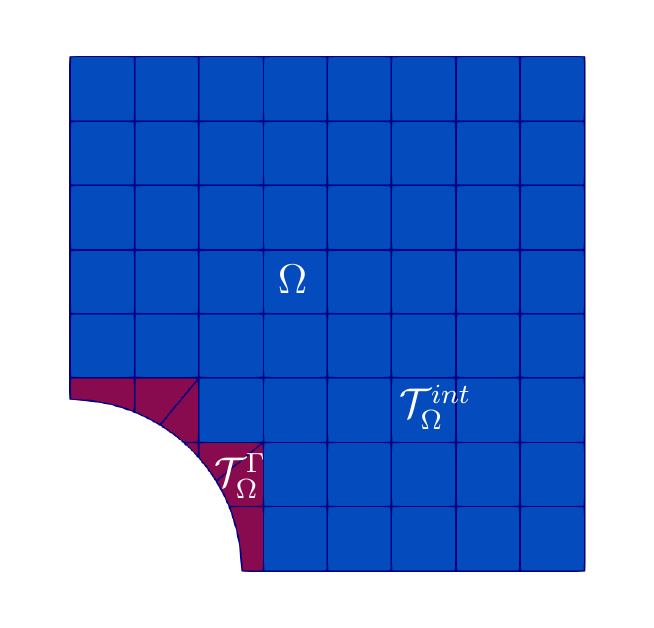}\fi}
 \caption{Infinite plate with hole problem: definition and trimmed computational domain.}
 \label{fig:plate_problem}
\end{figure}
We refer the reader to \cite{Timoshenko1951} for further details.
The computational domain is built as the boolean difference of a $L$-length square domain 
minus a circle with radius $R$.

In the Figure \ref{fig:plate_with_hole_l2_h1_area_perimeter} the solution errors in the
$H^1$ norm for different mesh sizes and degrees are shown,
together with the absolute errors of the  area 
computation.
In these results, the trimmed elements are
re-parameterized with the same degree of the solution discretization and using
the coarsest possible re-parameterization (see Figure \ref{fig:poisson_2D_reparam} as a matter of example). Thanks to the underlying Cartesian mesh, the geometric precision is down to $10^{-15}$.
The solution errors converge optimally for all the degrees,
and it is also the case for the area error, that presents the same odd/even
behavior described in previous section. 
\begin{figure}
 \centering
 \subfigure[Solution error in the $H^1$ norm.]
 {\ifdraftversion\else\includegraphics{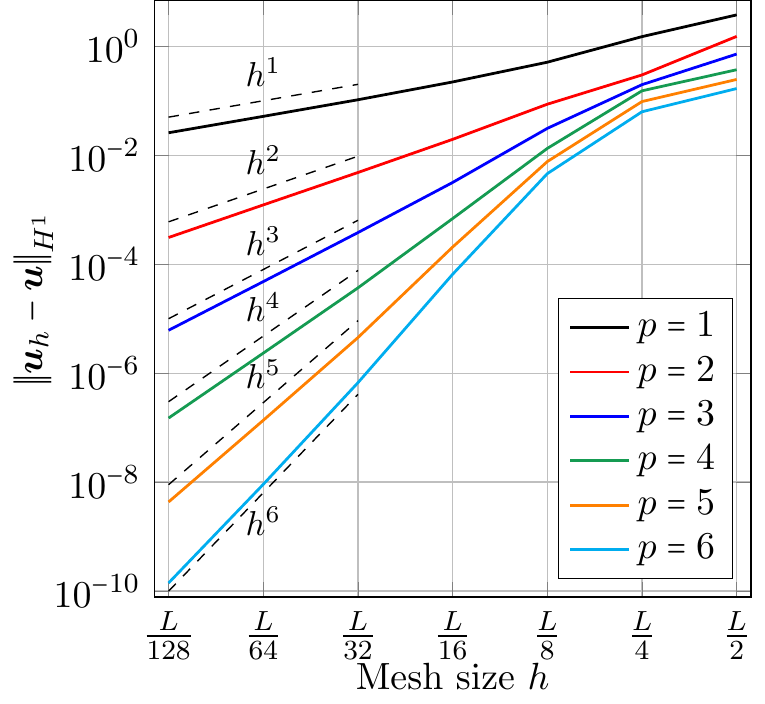}\fi}\hfill
 \subfigure[Area error.]
 {\ifdraftversion\else\includegraphics{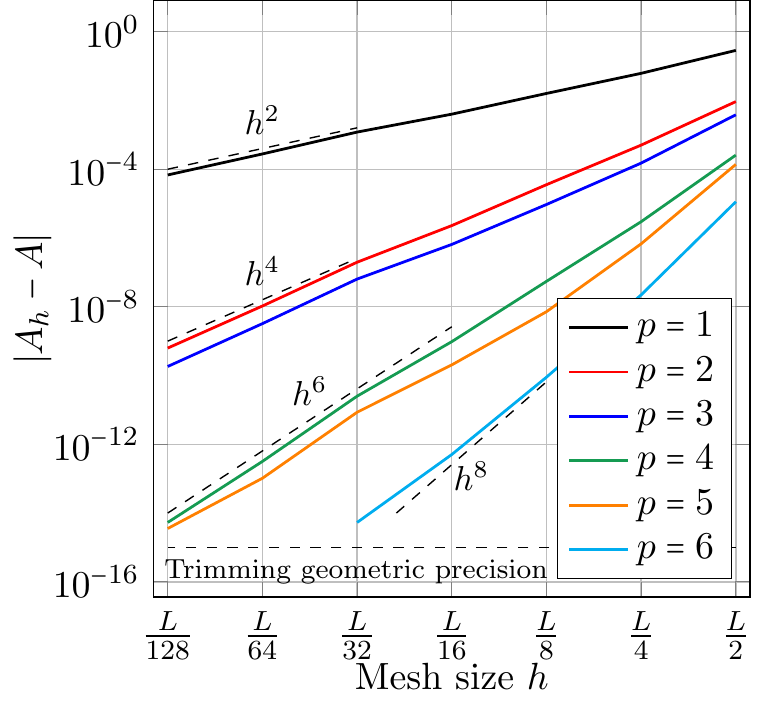}\fi}
 \caption{Infinite plate with hole problem: error of the solution $u_h$ in $H^1$ norm, and absolute error of area computation, for different
 degrees.}
 \label{fig:plate_with_hole_l2_h1_area_perimeter}
\end{figure}

Figure \ref{fig:plate_with_hole_conditioning} reports the condition number of the stiffness matrices
for the preconditioned (using the diagonal scaling preconditioner presented in Section \ref{sec:conditioning})
and non-preconditioned cases. As it was already seen in the Poisson example presented in Section~ \ref{sec:poisson_2D},
the diagonal scaling preconditioner keeps the matrix conditioning under control for all considered degrees.
\begin{figure}[b!]
 \subfigure[Without preconditioning.]
 {\ifdraftversion\else\includegraphics{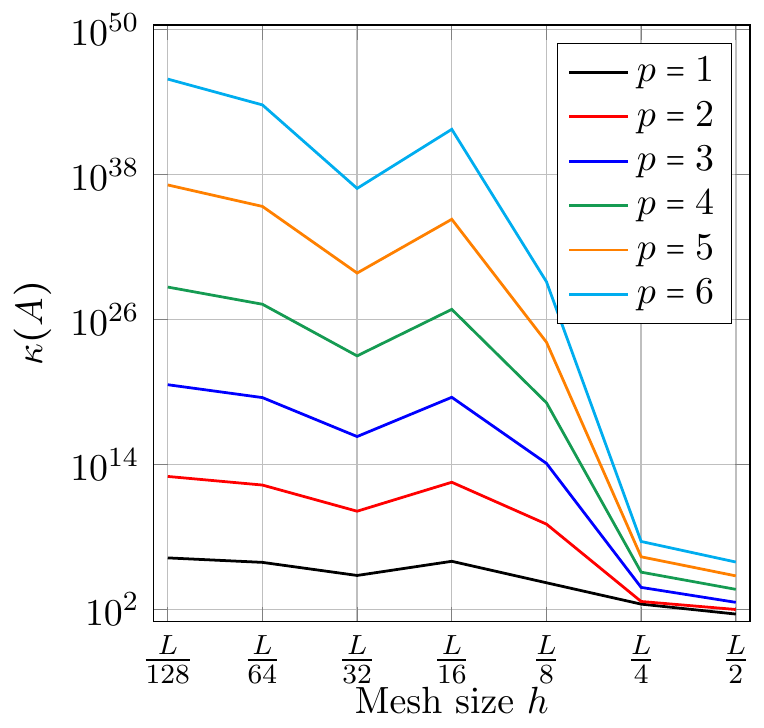}\fi}\hfill
 \subfigure[With diagonal scaling preconditioning.]
 {\ifdraftversion\else\includegraphics{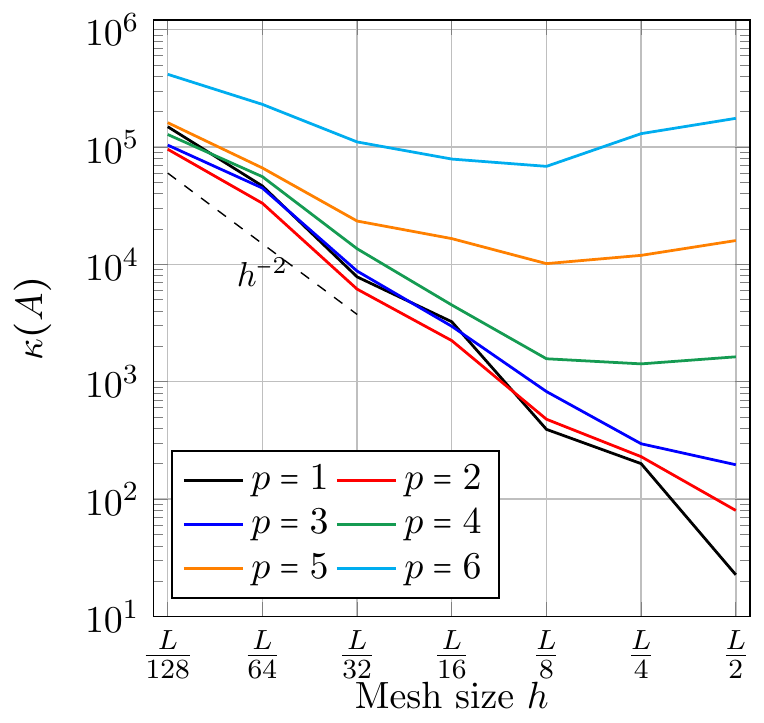}\fi}
 \caption{Infinite plate with hole problem: stiffness matrix conditioning with and without diagonal scaling preconditioning.}
 \label{fig:plate_with_hole_conditioning}
\end{figure}

\subsection{Poisson equation in a 3D domain} \label{sec:poisson_3D}
A Poisson problem in the interior of a 3D sphere is studied in this section:
\begin{subequations}
\begin{alignat}{2}
\Delta u &= f\quad&&\text{in }\Omega\,,\\
\nabla u\cdot\bm{n} &= g_n\quad&&\text{on }\partial\Omega\,.
\end{alignat}
\end{subequations}
The problem analytic solution, in spherical coordinates, is:
\begin{align}
  u(r,\,\theta,\,\varphi) &= r^2\,\sin^2\left(\frac{\pi}{R}\,r\right)\,\cos\theta\,\sin^2\varphi\,,
\end{align}
where $\theta$ is the azimuth and $\varphi$ the co-latitude.
The laplacian of $u$ is
\begin{align}
\begin{split}
  \Delta u(r,\,\theta,\,\varphi) &= 
	\left[\left(6-\frac{1}{\sin^2\varphi}\right)\,\sin^2\left(\frac{\pi}{R}\,r\right) + \frac{6\pi}{R}\,r\,\sin\left(\frac{2\pi}{R}\,r\right)
+2\left(\frac{\pi}{R}\right)^2 r^2\,\cos\left(\frac{2\pi}{R}\,r\right)\right]\\
&\cos\theta\,\sin^2\varphi + \left(2\,\cos2\varphi + \frac{\sin2\varphi}{\tan\varphi}\right)\,\sin^2\left(\frac{\pi}{R}\,r\right)\,\cos\theta\,,
\end{split}
\end{align}
where the Neumann condition on the sphere boundary is $\nabla u\cdot\bm{n} = g_n = 0$.
As in Section \ref{sec:poisson_2D},  we impose weakly the condition
\begin{align}
\int_\Omega u = \int_\Omega r^2\,\sin^2\left(\frac{\pi}{R}\,r\right)\,\cos\theta\,\sin^2\varphi=0\,.
\end{align}

The computational domain is built as the boolean intersection of a unit cube domain
and a sphere with radius $R=1$.

As it can be seen in Figure \ref{fig:poisson_3D_l2_h1_volume}, optimal approximation
properties are achieved both in
$L^2$ and $H^1$ norms for different mesh sizes and degrees.
The computation of the domain volume is also optimal and
presents the same odd/even behavior observed before.
\begin{figure}
 \subfigure[Solution error in the $L^2$ norm.]
 {\ifdraftversion\else\includegraphics{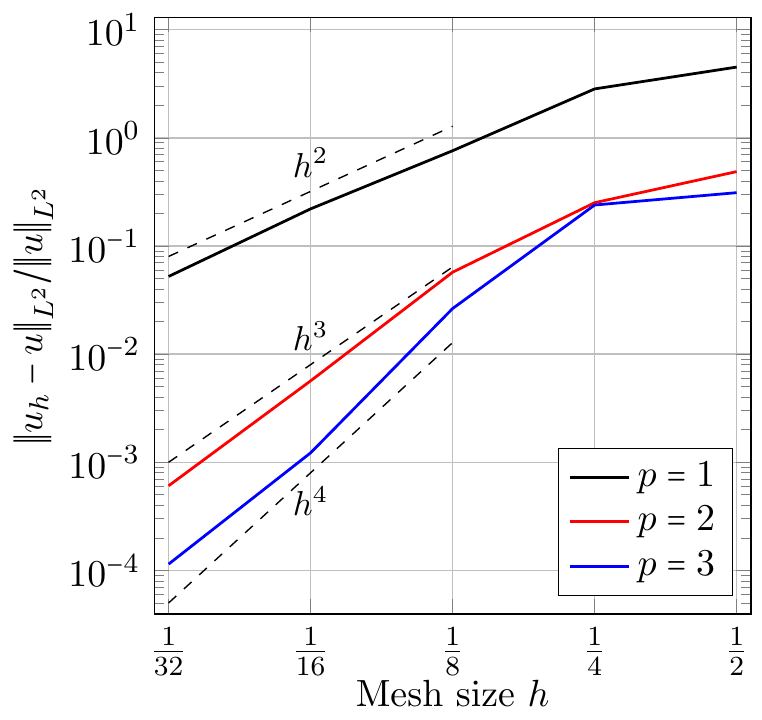}\fi}\hfill
 \subfigure[Solution error in the $H^1$ norm.]
 {\ifdraftversion\else\includegraphics{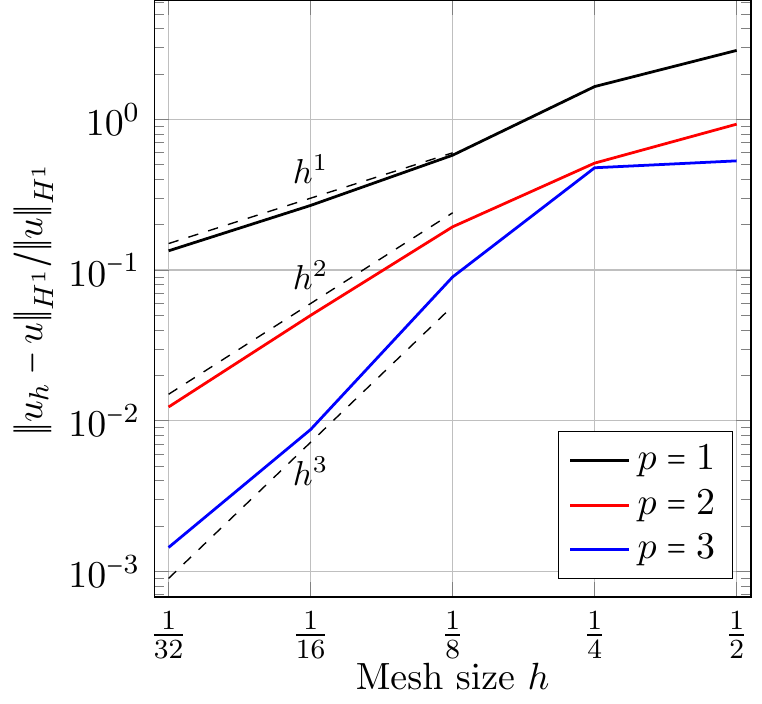}\fi}\\
 \centering
 \subfigure[Volume error.]
 {\ifdraftversion\else\includegraphics{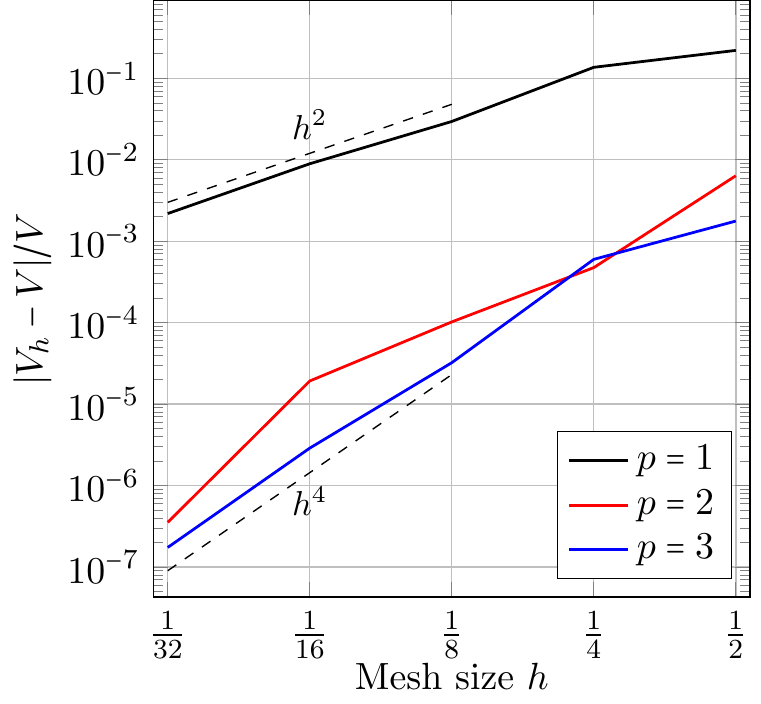}\fi}
 \caption{Poisson 3D problem: error of the solution $u_h$ in the $L^2$ and $H^1$ norms and volume error, for different
 degrees.} 
 \label{fig:poisson_3D_l2_h1_volume}
\end{figure}

Finally, and as in previous numerical experiments (Sections \ref{sec:plate_with_hole} and \ref{sec:poisson_2D}),
the matrix conditioning, with and without preconditioning, is analyzed in Figure \ref{fig:poisson_3D_conditioning}.
As before, the diagonal scaling preconditioner reduced drastically the matrix's condition number,
allowing to recover the expected asymptotic behavior  $h^{-2}$.
\begin{figure}
 \subfigure[Without preconditioning.]
 {\ifdraftversion\else\includegraphics{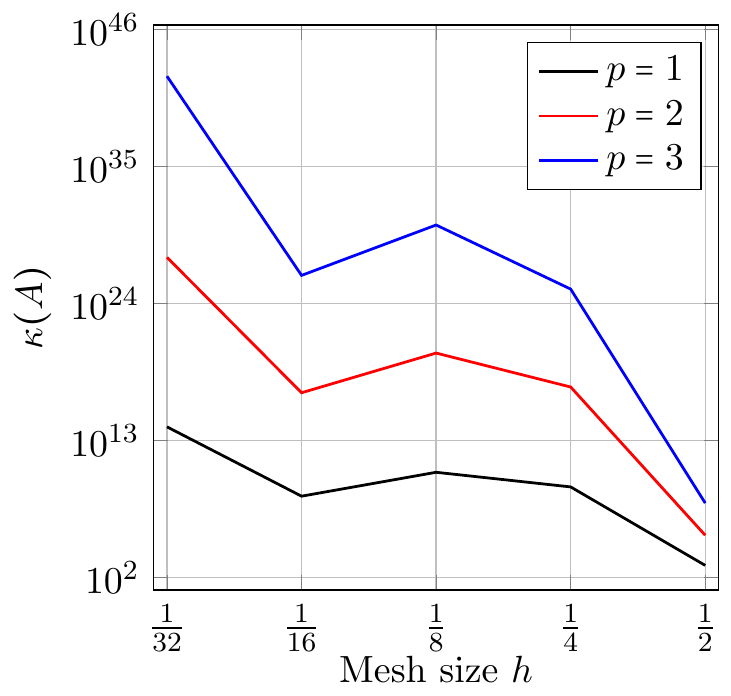}\fi}\hfill
 \subfigure[With diagonal scaling preconditioning.]
 {\ifdraftversion\else\includegraphics{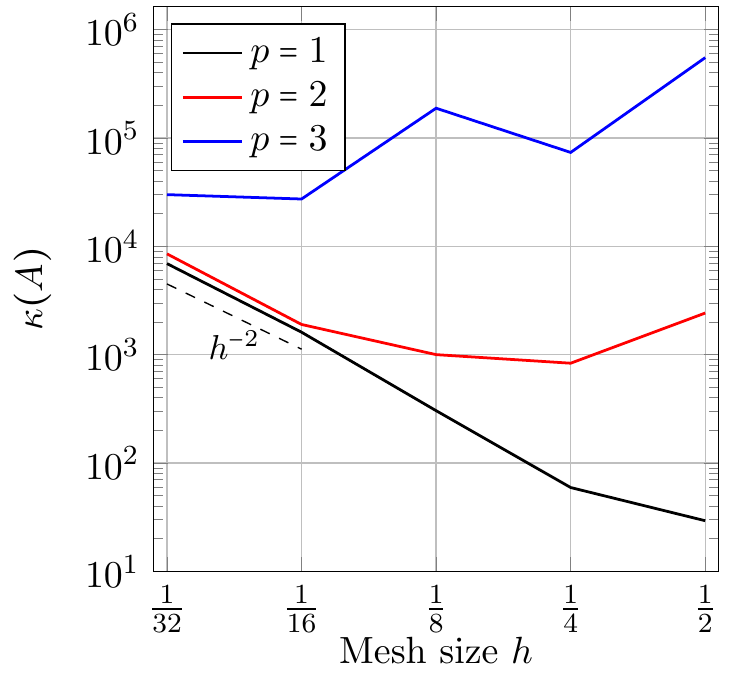}\fi}
 \caption{Poisson 3D problem: stiffness matrix conditioning with and without diagonal scaling preconditioning.}
 \label{fig:poisson_3D_conditioning}
\end{figure}

\clearpage
\subsection{A 3D mold with a cooling channel}
Finally, in order to illustrate the capabilities of the presented method for general geometries
in this last numerical example we study the elastic response of a 3D printed mold.
The mold presents an interior cooling channel whose mission is to reduce its temperature by means of a circulating cold fluid.

As it can be seen in Figure \ref{fig:mold_geom}, the cooling channel follows a complex path in the interior of the mold in order
to successfully refrigerate it: in enters from the interior cylindrical face,
travels throw the interior and exterior faces of the mold, and exists from the exterior cylindrical face.
In addition, some small cylindrical holes are also present in the top face of the mold (2 in the interior wall and 4 in the exterior one).
\begin{figure}[h]
 \includegraphics[width=0.49\textwidth]{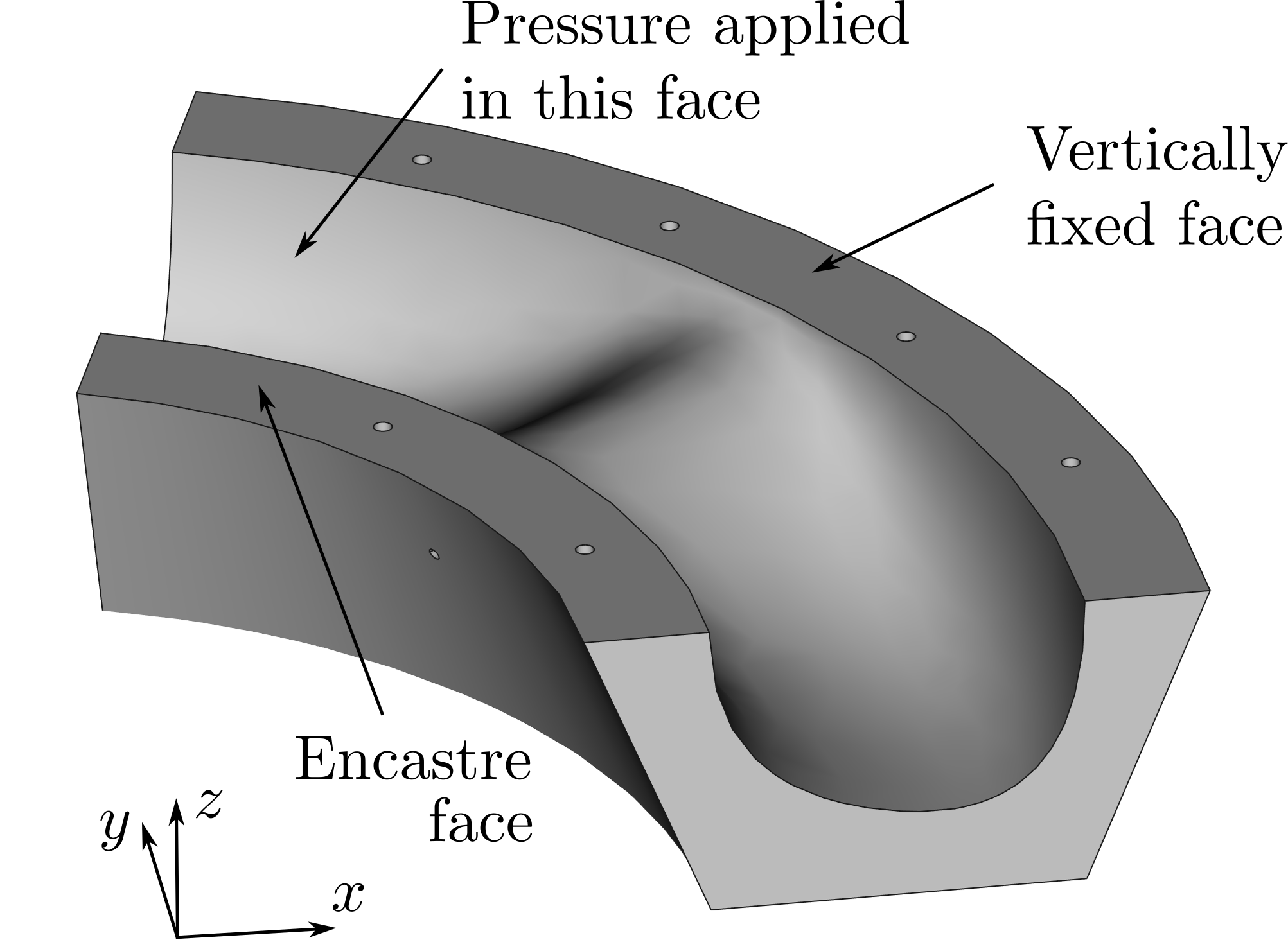}\hfill
 \includegraphics[width=0.49\textwidth]{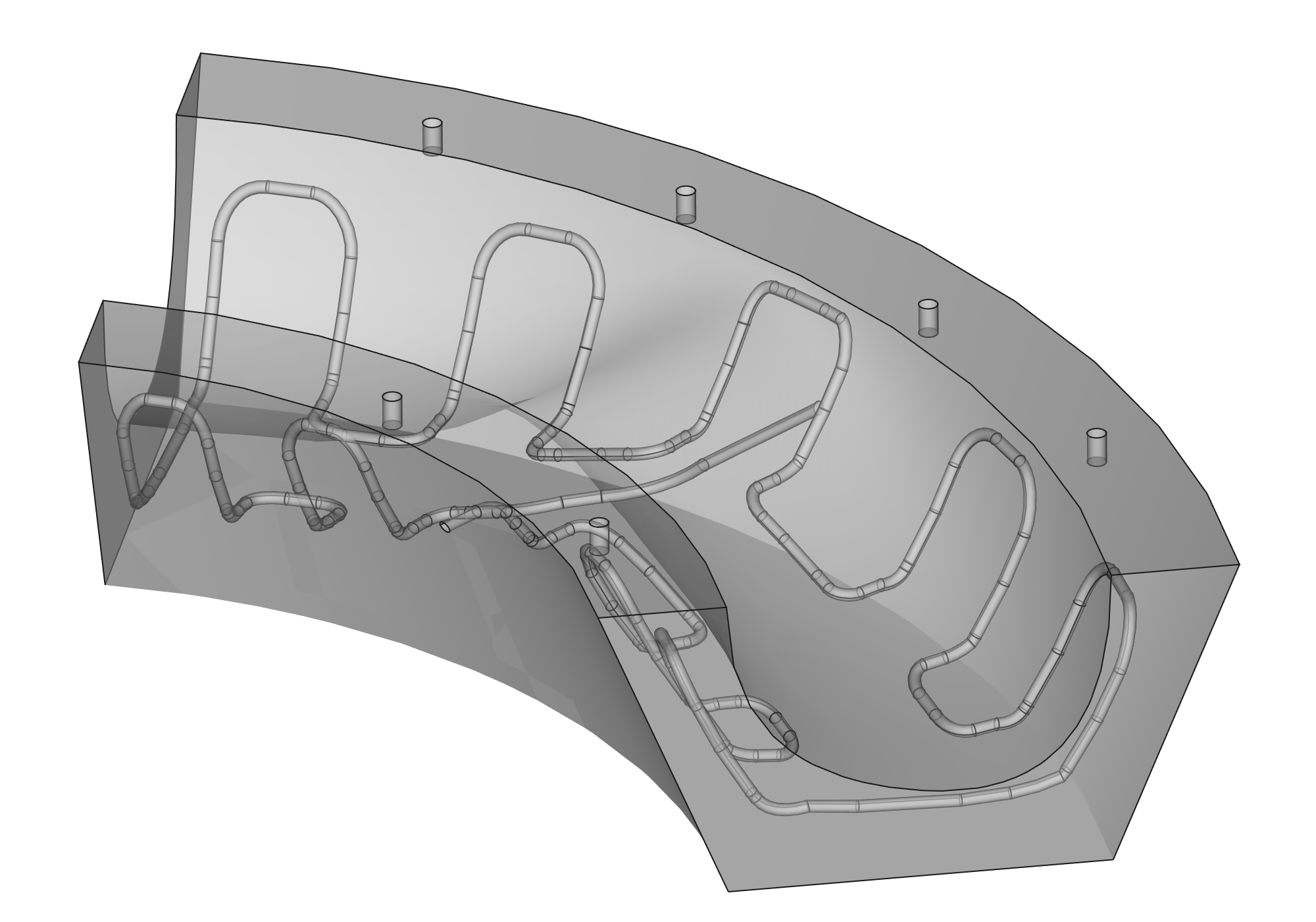}
 \caption{Mold with an interior cooling channel. The mold is a single trivariate and the cooling channel is defined as a boolean operation (subtraction).}
 \label{fig:mold_geom}
\end{figure}

The mold geometry is described with a single spline trivariate whereas
the cooling channel has been created as the extrusion of a circle along the channel path.
The final geometry is built by means of a boolean operation: the mold minus the cooling tube.
In Figure \ref{fig:mold_trim_elements} the non-trimmed elements (in green) together with the trimmed ones (in blue) are shown.
\begin{figure}
\subfigure[Trimmed (blue) and non-trimmed (green) elements of the trivariate after the trimming operation.]{
 \includegraphics[width=0.49\textwidth]{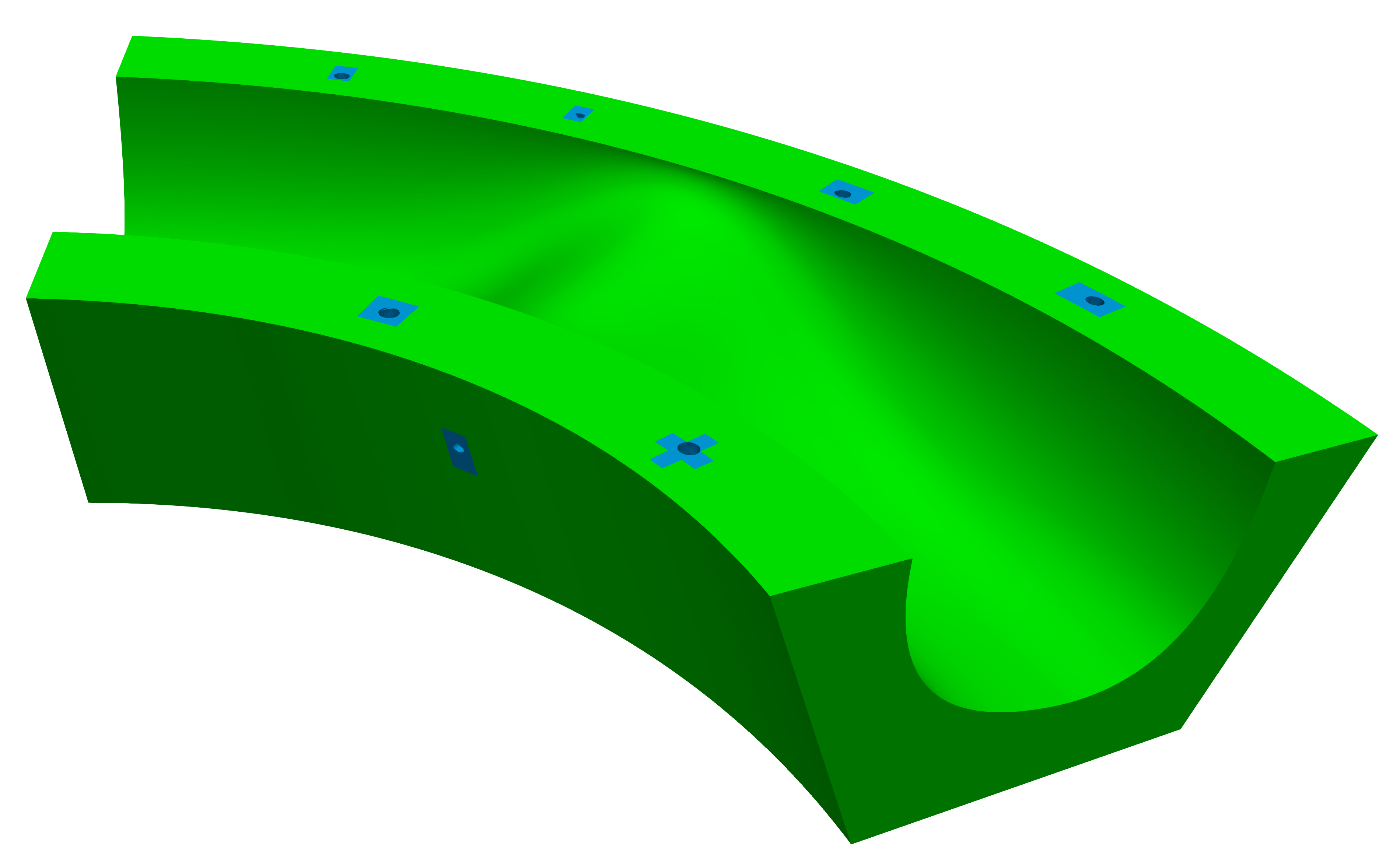}\hfill
 \includegraphics[width=0.49\textwidth]{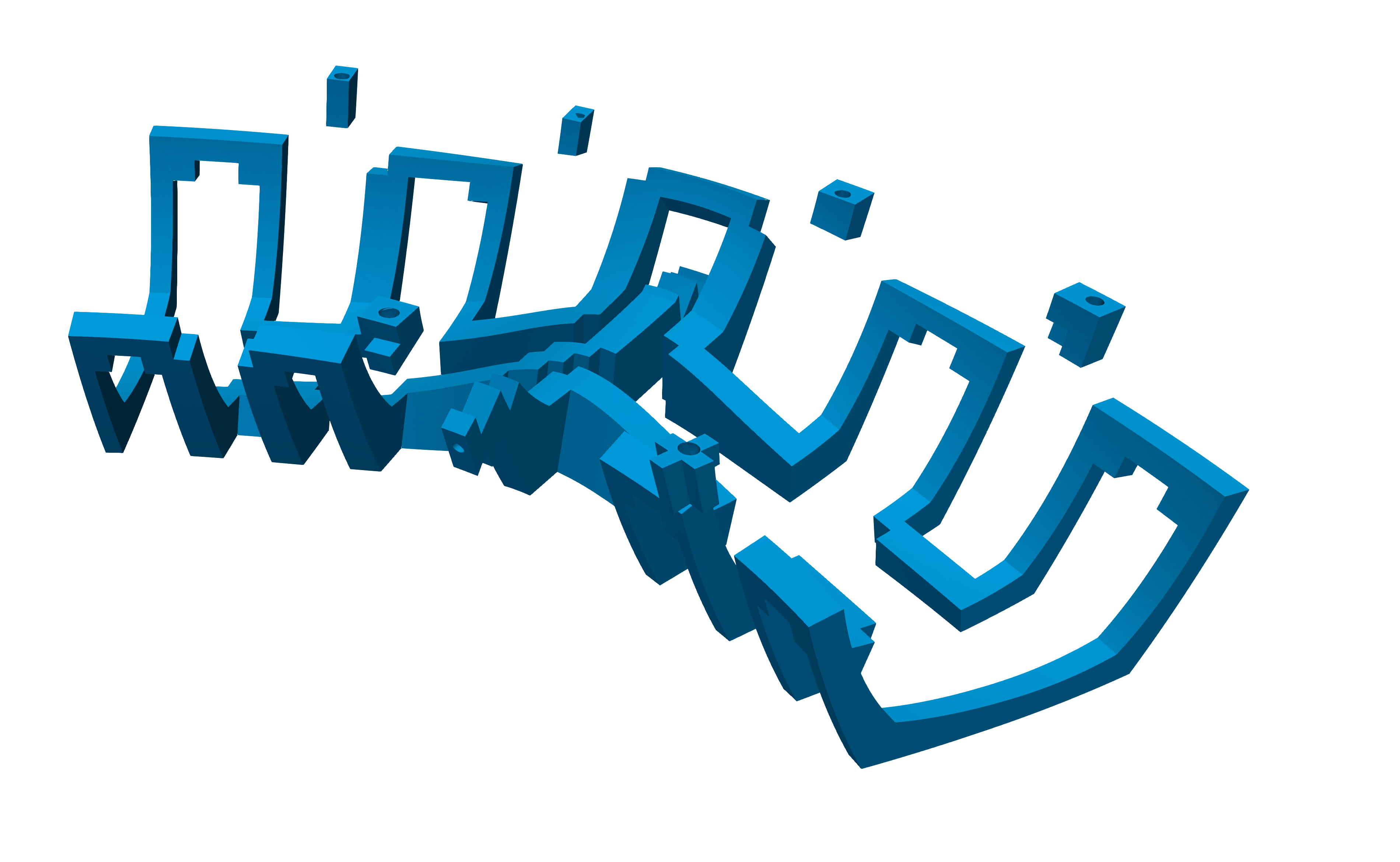}}\newline
\subfigure[Re-parameterization of some trimmed B\'ezier elements with cubic tetrahedra. Every color is associated to a different element.\label{fig:mold_reparam}]{
 \includegraphics[width=0.49\textwidth]{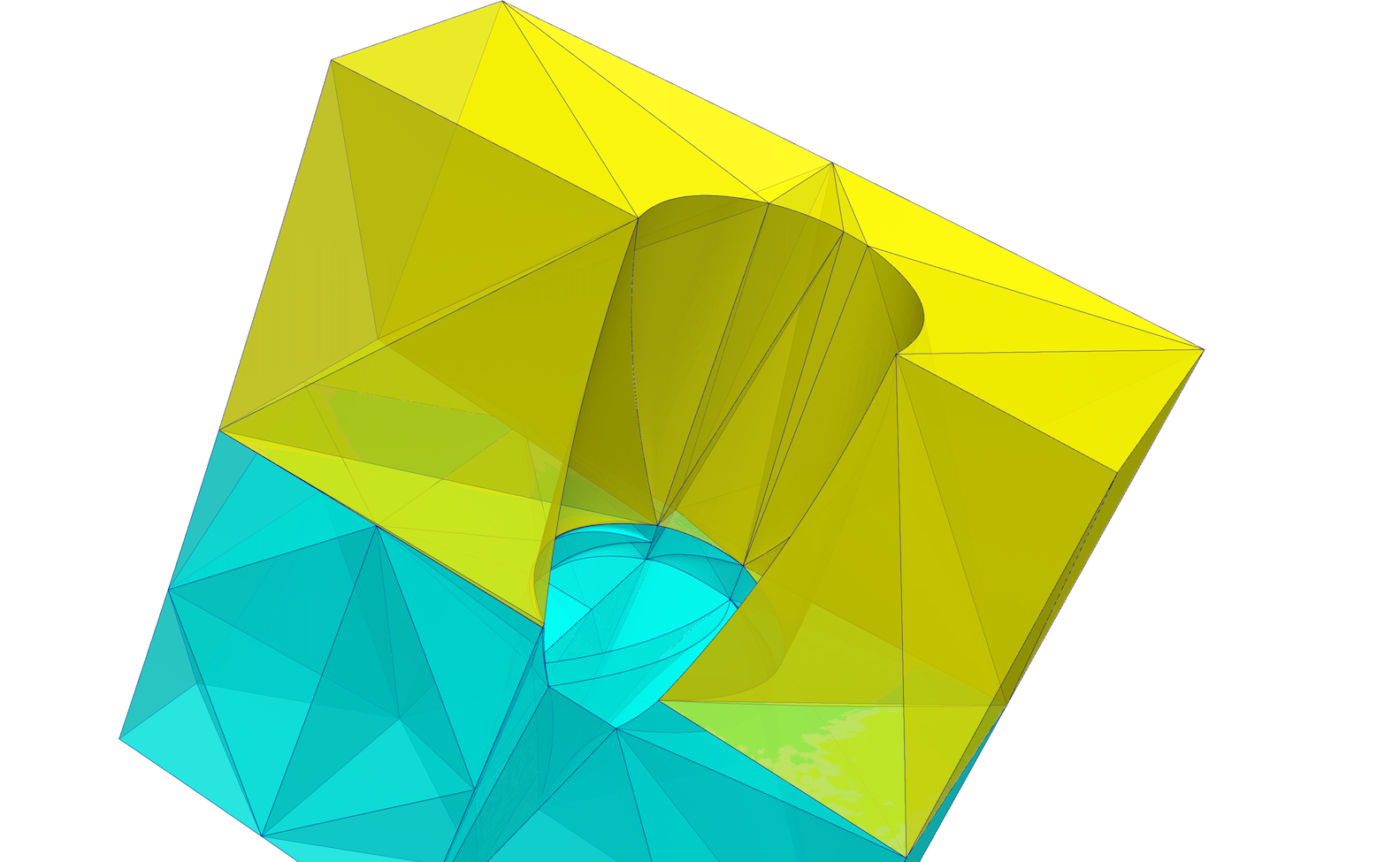}\hfill
 \includegraphics[width=0.49\textwidth]{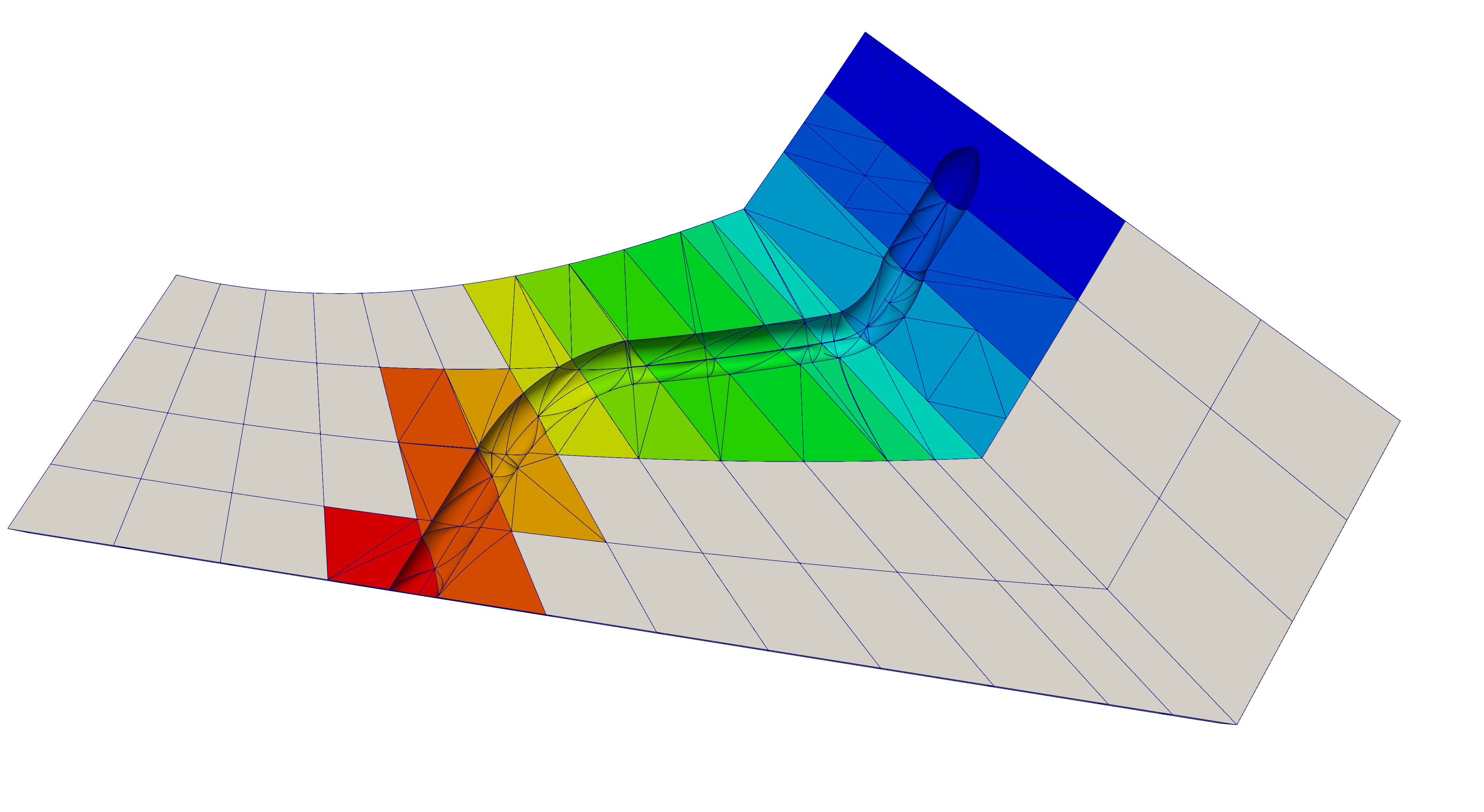}}
 \caption{Re-parameterization of the trimmed B\'ezier elements in the domain.}
 \label{fig:mold_trim_elements}
\end{figure}
The trimmed B\'ezier elements are re-parameterized with cubic tetrahedra.
As a matter of example, some re-parameterized elements are shown in Figure \ref{fig:mold_reparam}.

The trivariate has degree 3 along the longitudinal direction and degree 2 in the other two directions ($56\times6\times36$ elements).
The mechanical properties of the material mold are $E=10\,\text{GPa}$ and $\nu=0.2$.
These values are similar to the material properties of a metal specimen manufactured with an additive manufacturing procedure: they were inferred from \cite[Table 3]{Gu2013}.

The performed simulation consists in a linear elasticity analysis in which a traction force (Neumann condition) is applied on the curved face of the mold.
The internal top face is fully fixed encastre while the exterior top face is fixed just vertically.
The value of the traction is $1.0\,\text{MPa}$ along $z$-direction (pointing downwards) and $0.5\,\text{MPa}$ along $x$-direction.

The deformed mold together with the stress magnitude can be seen in Figure \ref{fig:mold_stress}.
The discontinuity that can be appreciated in the stress field near the corners is due to the fact that geometry parameterization is only $C^0$ continuous at those points.
Additionally, a view of the stress in the interior of the mold is shown in Figure \ref{fig:mold_stress_details}. As it can be seen, despite the fact that the mesh size is similar to the channel diameter, the stress distribution is affected by the presence of the channel.
\begin{figure}
 \centering
 \subfigure[Exterior stress distribution.\label{fig:mold_stress}]
 {\includegraphics[width=0.8\textwidth]{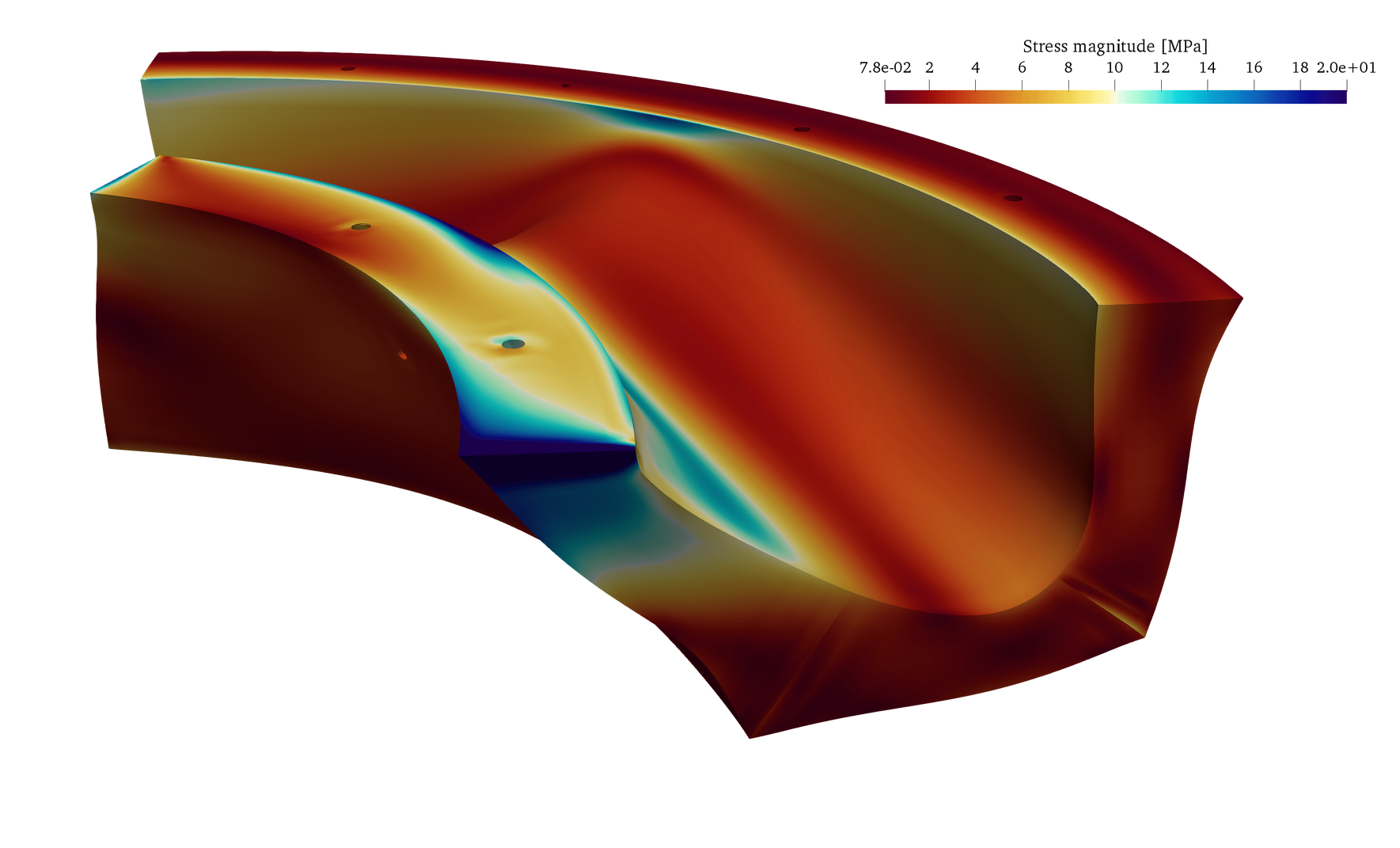}}\newline
 \subfigure[Interior stress along some a longitudinal section. Both pictures correspond to the front and back views of a one element thick slice.\label{fig:mold_stress_details}]
 {\includegraphics[width=0.8\textwidth] {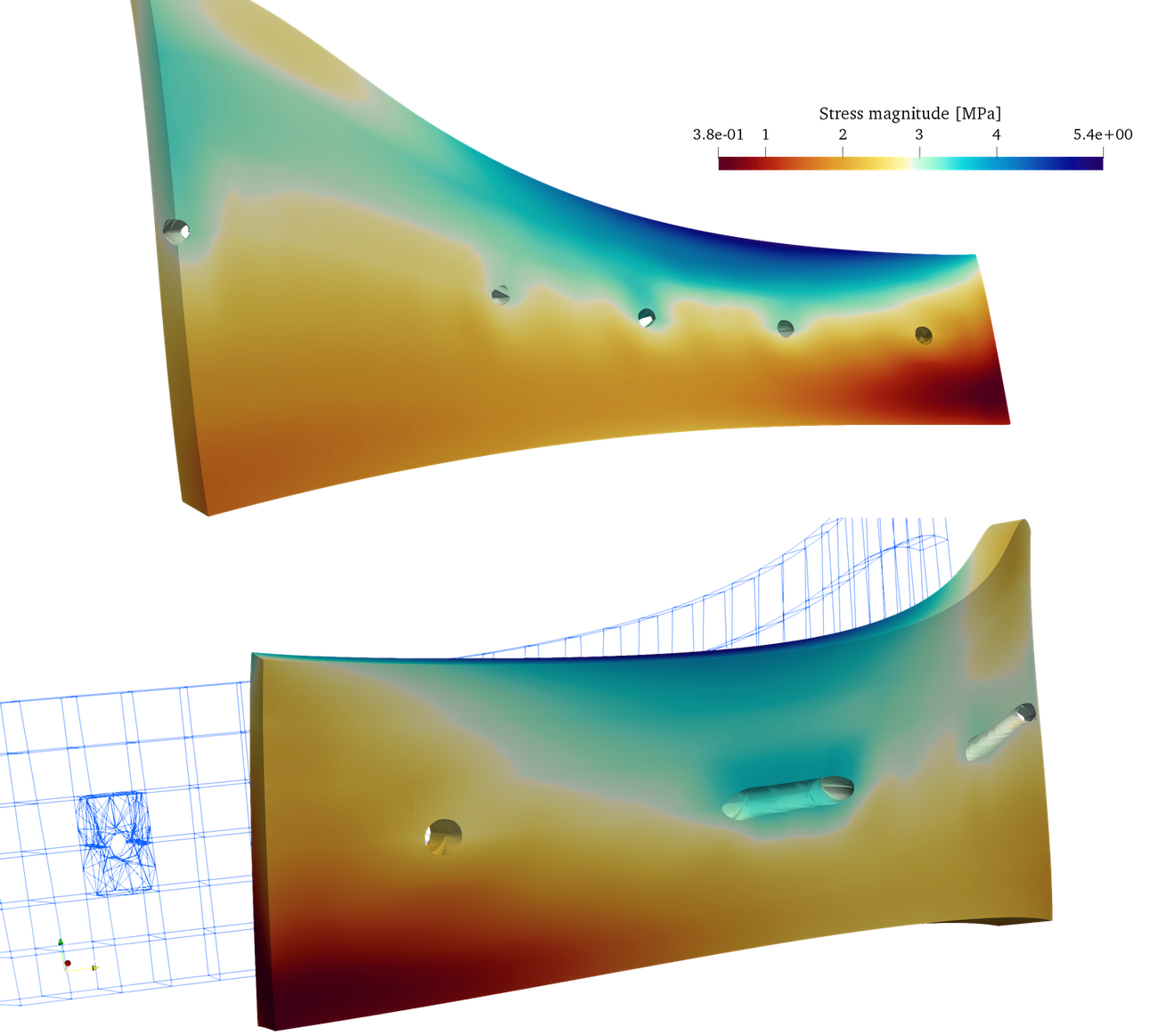}}
 \caption{Stress magnitude distribution of the mold plotted on the deformed configuration of the geometry.}
\end{figure}

\section{Conclusions} \label{sec:conclusions}

We have presented a novel approach for the construction of numerical methods for elliptic PDEs on trimmed geometries, seen as a special class of more general V-reps.
\review{More specifically, this method is applicable to the case of trimmed non-conforming multipatch trivariate volumes.

Our approach is based on the local re-parameterization of B\'ezier elements that are cut by the trimming, in contrast to the creation of a global re-parameterization of the full domain.
Thus, by performing a re-parameterization of the trimmed elements with the same degree as the discrete solution, we proved that our method guarantees optimal convergence rates when possibly non-homogeneous natural (Neumann) boundary conditions are imposed on the part of the boundary affected by trimming.

This theoretical result is supported by numerical experiments that illustrate how the use of low-order re-parameterizations lead inevitably to a lack of optimal convergence.
The use of finer, but still low-order element re-parameterizations, mitigates this problem for coarse solution mesh sizes $h$. But, for sufficiently small values of $h$, fine low-order element re-parameterizations yield a sub-optimal behavior.

Our numerical experiments also illustrated how the limited precision values inherent to the available geometric modeling tools limit the potential accuracy achievable by our numerical methods.

In all the considered experiments the use of a diagonal scaling preconditioner proved itself efficient in controlling the ill-conditioning derived from the trimming of the basis functions' support.

Regarding the re-parameterization of the trimmed elements, in this work we proposed a novel methodology for building adapted quadrature schemes based on the creation of coarse high-order tetrahedral and / or hexahedral finite element meshes. The only requirement to these meshes is the absence of sign change in tetrahedra / hexahedra Jacobians.
The potential of this methodology was in the local re-parameterization of a complex V-rep.
}

Finally, we would like to remark that the imposition of essential (Dirichlet) boundary conditions on trimmed boundaries ask for stabilization: in this case our approach should be combined with the stabilization proposed and analyzed in \cite{buffa_minimal_2019} and this is object of further studies. 
\review{The extension of the presented method to the case of domains constructed as the union of V-reps will be addressed in the future.}


\section*{Acknowledgements} 
Pablo Antolin and Annalisa Buffa gratefully acknowledge the support of the European Research Council, through the ERC AdG  n. 694515 - CHANGE. 
Massimiliano Martinelli has been supported by the European Union's Horizon 2020 research and innovation programme under grant agreement n.~680448 - CAxMan.
The authors also acknowledge R.~V\'azquez for his helpful insights and suggestions, and would like to thank G.~Elber and F.~Massarwi for the many  inspiring discussions on volumetric modeling. 

\bibliographystyle{plain}
\bibliography{biblio}

\end{document}